\documentclass[11pt]{amsart}

\usepackage[margin=1in]{geometry}
\usepackage[
colorlinks=true, citecolor=blue]{hyperref}
\usepackage{amsmath, amssymb, amsthm, esint,  mathtools, mathrsfs, bbm}
\usepackage{enumerate, subcaption, graphicx}
\usepackage{xcolor}
\usepackage[T1]{fontenc}
\usepackage[utf8]{inputenc}
\usepackage{geometry}
\usepackage{graphicx}
\usepackage{enumitem}
\usepackage{lmodern}
\usepackage{cancel}
\usepackage{listings}%
\usepackage{booktabs}%
\usepackage{algorithm}%
\usepackage{algorithmicx}%
\usepackage{algpseudocode}
\usepackage{textcomp}%

\usepackage[normalem]{ulem}
\newcommand{\cred}{\color{red}}

\newcommand{\bu}{\boldsymbol{u}}
\newcommand{\bw}{\boldsymbol{w}}
\newcommand{\bv}{\boldsymbol{v}}
\newcommand{\bH}{{\bf H}}
\newcommand{\bq}{{\bf q}}

\newcommand{\bL}{{\bf L}}
\newcommand{\bn}{{\boldsymbol{n}}}
\newcommand{\bftau}{\boldsymbol{\tau}}
\newcommand{\bfphi}{\boldsymbol{\phi}}

\newcommand{\bfeta}{\boldsymbol{\eta}}
\newcommand{\bfxi}{\boldsymbol{\xi}}
\newcommand{\bfzeta}{\boldsymbol{\zeta}}
\newcommand{\bA}{{\bf A}}

\newcommand{\sO}{\mathcal{O}}
\newcommand{\Pin}{{P}_{in}}
\newcommand{\Pout}{{P}_{out}}
\newcommand{\Pinout}{{P}_{in/out}}
\newcommand{\sB}{\mathcal{B}}

\newcommand{\blue}[1]{\textcolor{blue}{#1}}
\newcommand{\green}[1]{\textcolor{green}{#1}}
\let\cal\mathcal
\allowdisplaybreaks[1]

\theoremstyle{thmstyleone}%
\newtheorem{theorem}{Theorem}%
\newtheorem{proposition}[theorem]{Proposition}%
\newtheorem{lemma}[theorem]{Lemma}%

\title[Finite-time contact: Navier-slip]{Finite-time contact in fluid-elastic structure interaction: Navier-slip coupling condition}
\author[K. Tawri, N. Ward]{Krutika Tawri$^1$, Nash Ward$^1$} 

\address{	\newline
$^1$ Department of Applied Mathematics, University of Washington, WA, USA.}
	
\email{ ktawri@uw.edu (Krutika Tawri), Nashw1@uw.edu (Nash Ward)}

\date{} 

\begin{document} 
\begin{abstract}
    We consider a fluid-structure interaction problem involving a viscous, incompressible fluid flow, modeled by the 2D Navier-Stokes equations, through a thin deformable elastic tube, displacement of which is not known a priori. The elastodynamics problem is given by 1D plate equations. The fluid and the structure are nonlinearly coupled via the kinematic and dynamic coupling conditions at the fluid-structure interface. The fluid flow is driven by dynamic pressure data imposed at the inlet and outlet of the tube. We impose the Navier-slip boundary condition at the deformable fluid-structure interface and at the bottom rigid boundary of the fluid domain. Hence, beyond the usual geometric nonlinearities arising from nonlinear coupling in FSI with no-slip, the analysis is more challenging due to the possibility of tangential jumps of the fluid and structural velocities at the moving interface. We first discuss the existence of weak solutions and then establish a `hidden' spatial regularity result for the structure displacement.

  Our main result proves the existence of a finite time for weak solutions at which the compliant upper boundary meets the lower boundary (i.e., the tube collapses), provided that there is a sufficient pressure drop across the channel. This resolves the ``no-collision'' paradox identified by Grandmont and Hillairet in the no-slip setting in [{\it Arch. Ration. Mech. Anal., 220(3): 1283-1333, (2016)}], the counterpart to the present work. To the best of our knowledge, this is the first work that rigorously establishes finite-time contact in a fluid-elastic structure interaction system, thereby validating the model to correctly capture near-contact dynamics.
\end{abstract}
\maketitle

\section{Introduction}

In this article, we study the behavior of weak solutions to a fluid-structure interaction model.
The incompressible flow of the viscous fluid is described by the 2D Navier-Stokes equations, while the thin (visco-)elastic membrane is characterized by plate equations. The fluid and the structure are fully coupled across the moving fluid-structure interface through a two-way coupling:
(i) the dynamic coupling condition ensures continuity of  contact forces at the interface, and (ii) the kinematic coupling condition ensures the continuity only of the normal components of their velocities. In other words, we impose a Navier-slip condition at the moving interface, permitting relative motion between the fluid and the structure in directions tangential to the current configuration of the fluid domain. We prescribe dynamic pressure (Bernoulli) conditions at the inlet and the outlet of the elastic tube; we also prescribe a constant transversal flow rate condition {\it only} at the outlet of the channel.
The fluid flow is thus driven by a pressure drop across the channel. We assume that the elastic plate is clamped at both ends. This model, incorporating the slip of the viscous fluid across the structure, was first proposed in \cite{MC16}.

There has been a considerable amount of work done concerning well-posedness analysis of fluid-elastic structure interaction (FSI) over the past two decades, specifically in the setting where no-slip kinematic coupling condition, that ensures the continuity of the fluid and structure velocities at their interface, is considered. More precisely, see \cite{CS06, CS10, KTZ10} and the references therein, for results on local-in-time existence of regular solutions under different settings, and \cite{CDEG, G08, MC13, LR14} for local-in-time existence of weak solutions. While some works address collisions by constructing special global-in-time weak solutions that are continued past collisions (see, for instance, \cite{CGH21} for elastic bodies and \cite{F03, SST02} for rigid bodies), most well-posedness results establish existence only up to the point of collision. In particular, for the aforementioned fluid-elastic structure interaction results, existence is given until the point when the elastic boundary makes contact with the rigid bottom boundary of the fluid domain.

The 2016 article \cite{GH_noslip} gave a global existence result of strong solutions to an FSI problem consisting of a 1D viscoelastic beam and a 2D viscous incompressible fluid, under the no-slip coupling condition, and revealed, for the first time, that contact
of the walls of the compliant tube in finite time is not possible. 
This is the famous `no-contact' paradox, also known as the {\it d'Alembert or Cox-Brenner paradox}. In the no-slip case, the no-contact result given in  \cite{GH_noslip} was extended to 3D fluid-2D structure interaction setting, recently in \cite{BKLM25}, under certain assumptions on the regularity of the structure.

These works highlight the fact that, although the no-slip is a common assumption in the blood flow literature and more broadly in general FSI applications, the slip condition is considered to be a more realistic boundary condition in modeling near-contact dynamics, such as the closure of heart valves, because it allows for the possibility of collisions and self-intersections (see e.g. \cite{NP10}, \cite{MC16}).
Despite this, the Navier-slip coupling condition, in moving boundary problems involving elastic structures, has been far less explored.
The works \cite{MC16, MC19} and \cite{T24} establish the existence of a weak solution for the 2D fluid-1D structure model incorporating the Navier-slip coupling condition at the fluid-structure interface, considered in the present study, both in deterministic and stochastic frameworks, up to the point of tube collapse. 
It should be emphasized that both of the aforementioned studies \cite{MC16, T24} consider a {\it purely elastic} structure (no structural viscosity) and make no assumptions restricting the motion of the structure, allowing deformations in either direction. While there are several difference in the analysis carried out in \cite{MC19} and \cite{T24}, we particularly note that the compactness results established in \cite{MC19} and \cite{T24} rely on different approaches, and point out that in this manuscript, we employ a version of the compactness argument introduced in \cite{T24}.

A substantial body of work explores fluid-{\it rigid} (immersed) body interactions under both slip and no-slip kinematic conditions. Here, we focus only on those studies relevant to the slip case.
The local-in-time existence of weak solutions in 3D was first proven in \cite{GH14}. In \cite{GH_slip}, the authors studied the effect of different kinematic coupling and boundary conditions on contact problems. They proved that when slip boundary conditions are imposed both at the fluid-structure interface and at the bottom boundary of the fluid domain, a smooth (spherical) rigid body immersed in a viscous fluid and subjected to gravity will come into contact with the bottom boundary in finite time. Whereas, if either the fluid-structure interface or the bottom boundary enforces a no-slip condition, the rigid body cannot make contact in finite-time, regardless of the relative densities of the fluid and the structure.
This work builds on previous results regarding the calculation of drag forces experienced by a rigid body in simplified settings considering Stokes flow \cite{S04, GH12}.

While fluid-elastic structure interaction under the Navier-slip condition has historically received limited attention, interest in these models has been growing in recent years, highlighting the need for further investigation. In the recent work \cite{MR26}, the authors study an FSI model with slip in 3D under the assumption that the structure is only allowed to have transversal deformations, and prove the existence of weak solutions until the point of contact. Another recent work \cite{CKS26} considers largely deforming viscoelastic bulk solid immersed in a viscous, incompressible fluid, and also proves the existence of weak solutions (defined two ways) until the point of contact. 
There has also been significant interest in computationally understanding and capturing the contact, no-contact or contactless rebound phenomenon in different FSI settings in the recent years; see e.g. \cite{GSST22, BFFG22}. However, the analysis is at a very early stage.

In this article, due to the imposition of a non-standard inlet-outlet boundary condition, we first provide a brief discussion on the existence of weak solutions for the fluid-structure system with Navier-slip coupling conditions at both the moving interface and the bottom boundary of the fluid domain. The weak solution is defined using fluid-structure joint test functions that allow a jump in the direction tangential to the fluid-structure moving interface.  We also assume that the thin clamped elastic structure deforms only in the direction normal to the reference geometry (which is assumed to be flat). While this condition is {\it not} required for the existence result, we impose this restricted structural deformation condition for the next part of our paper.
More precisely, the fluid flow is supplemented with inlet-outlet dynamic pressure conditions together with a constant flow rate condition at the outlet. 
Such boundary conditions, involving prescribed dynamic pressure data and additionally prescribed transversal flux rate (on partial boundary) have been studied in several previous works; see e.g. \cite{HRT96,S23} and references therein. This construction is based on a penalty method, in which we augment the original weak formulation with a term that penalizes deviations of the outlet total flux from unity. The addition of this penalty term introduces a {singular perturbation} to the original problem, creating challenges in establishing the compactness of the approximate solutions. To address this, we develop compactness techniques tailored to our penalization method and demonstrate that the fluid and structure velocities admit bounded fractional time derivatives, independently of the penalty parameter.

We then present our result on finite-time contact for the weak solutions constructed in the preceding section. Specifically, we show that when a pressure drop is maintained across the elastic channel for a sufficiently long duration, the upper compliant boundary of the fluid domain collapses onto the lower rigid boundary in finite time. To the best of our knowledge, this represents the {\bf first finite-time contact result} in fluid-structure interaction involving {\it elastic} structures

Our approach focuses on constructing an appropriate joint test function for the coupled fluid-structure system. This construction builds on the foundational idea of \cite{GH_noslip}, in which the fluid test function is obtained by approximating the solution of the steady Stokes equations posed on the frozen configuration of the moving domain at each instant in time. However, while \cite{GH_noslip} considers a rigid, smooth body, we work with an elastic structure, which is not necessarily smooth.
This distinction introduces several additional challenges. The treatment of the Navier-slip condition, which is formulated in terms of normal and tangential directions at the evolving interface, becomes significantly more involved in the elastic case, since these directions now evolve in time together with the deformation of the structure. {Moreover, in contrast to the rigid body case considered in \cite{GH_noslip}, the present setting allows the entire upper boundary of the fluid domain to be compliant. This may give rise to genuine topological degeneracies, as the collapse of the structure leads to loss of connectivity and fragmentation of the fluid domain.} Furthermore, several analytical tools, such as the trace theorem in a moving domain setting (previously studied in works such as \cite{CDEG, LR14}), are unavailable as the constants of continuity depend inversely on the minimum height of the fluid domain. The techniques of \cite{CGH21}, where fluid velocity is extended outside of the moving domain to deal with the cusps formed in the fluid domain at contact, do not apply either since we consider the slip conditions which {induce} discontinuity between the fluid and structure velocities at the interface.
To circumvent these issues, we develop {delicate, alternative} approaches that avoid relying on the aforementioned tools altogether.

Another key difference is that \cite{GH_noslip} enforces the contact to occur at only a single, predetermined point, an assumption that is physically plausible. Within our framework, such an assumption does not reflect the physical behavior of the system and therefore cannot be imposed. Indeed, the points corresponding to minimal structural height may be non-unique and can even form an infinite set, as the structure may exhibit highly oscillatory behavior. As a result, the localized analysis techniques developed in \cite{GH_noslip} are not directly applicable in our setting. Hence, we do not use the methods of \cite{GH_noslip} that rely on proving that the drag force experienced by the rigid body, on account of the viscous fluid, is not strong enough to counteract the forces of gravity. 
Instead, we construct a test pair that does not rely on such localized analyses or computation of drag. Our construction of test function pair involves the first spatial derivative of the structure displacement.
However, this requires the structure to have more regularity than afforded by the energy estimates. In particular, its curvature must be bounded in space. This regularity is not available for the weak solutions we consider in this manuscript. Hence, we first prove a `hidden' regularity result which improves the spatial regularity of structure displacement over the ``basic'' regularity provided by the energy estimate. This regularity result, does {\it not} involve bounds that are independent of the structure displacement. Instead, this result only claims that for as long as there is no contact, the structure belongs to the more regular space $L^2_tH^3_x$ (i.e. the bounds may blow up at contact). This is sufficient to justify the use of our constructed test functions, for as long as the structure is away from the bottom boundary.

The paper is organized as follows: In Section \ref{sec:setup} we describe the fluid-structure interaction model, derive the weak-formulation and the energy of the system. We then layout the mathematical framework by defining appropriate functional spaces and by giving the definition of weak solutions. Our main results are stated in Section \ref{sec:mainresults}. In Section \ref{sec:exist}, we discuss a proof of existence of weak solutions satisfying appropriate boundary conditions.
In Section \ref{sec:regularity}, we prove the `hidden' spatial regularity of the structure displacement and in Section \ref{sec:contact} we prove the finite-time contact result.


\section{Problem set up}\label{sec:setup}
This paper investigates the interactions between an elastic membrane and a two dimensional flow of a viscous, incompressible
Newtonian fluid. The fluid flows through a channel, denoted by $F_h$, which consists of four boundaries: The compliant boundary on top is denoted by $\Gamma_h$, the rigid bottom boundary is denoted by $\Gamma_b$ and the fixed inlet and outlet boundaries that permit fluid flow are denoted by $\Gamma_{in},\Gamma_{out}$ (see Fig. \ref{FSI}).
\begin{figure}[h]
    \centering
    \includegraphics[width=0.6\linewidth]{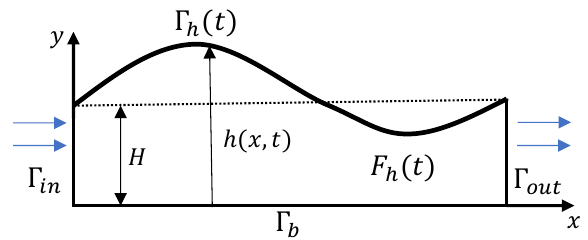}
    \caption{A snapshot of fluid domain}
    \label{FSI}
\end{figure}

For a continuous function $h(x,t):[0,L]\times [0,\infty)\mapsto \mathbb{T}$, we introduce the notation for fluid domain,
\begin{equation*}
    {F_h(t)} := \{ (x,y): 0< x < L, \,\, 0 < y < h(x,t)\}.
\end{equation*}
In this work, $h$ which determines the location of the elastic structure sitting atop the fluid domain, is itself an unknown in the problem.
The time-dependent (deformable) boundary of the fluid domain ${F_h(t)}$ is given by
$${ \Gamma_h(t) := \{(x,h(x,t)): x\in (0,L)\} ,}$$
whereas the rest of the fixed boundaries $\Gamma_b$, $\Gamma_{in}$ and $\Gamma_{out}$, denoting the bottom boundary and the inlet and outlet boundaries are as follows,
\begin{align*}
        \Gamma_b &:= (0,L) \times \{0\},\\
        \Gamma_{in} := \{0\} \times (0,H),&\quad
        \Gamma_{out} := \{L\} \times (0,H).
\end{align*}

The incompressible flow of a viscous fluid is governed by the Navier-Stokes equations. We look for fluid velocity $\bu =(u_1,u_2):{F_h(t)}\to \mathbb{R}^2$ and pressure $p:{F_h(t)}\to\mathbb{R}$ satisfying,
\begin{equation}
   \begin{split}\label{fluid_eqn} 
      &  \rho_f(\partial_t\bu +\bu\cdot\nabla \bu) - \text{div}\boldsymbol{\sigma}(\bu,p) = 0, \qquad \text{ in } {F_h(t)}\\
   & \text{div }\bu=0, 
\qquad \text{ in } {F_h(t)},
   \end{split}
\end{equation}
where $\displaystyle\boldsymbol{\sigma}(\bu,p) = \mu\nabla \bu - p{\bf I}$ is the Cauchy stress tensor. This constitutive law ignores the contributions of the symmetric part $(\nabla\bu)^T$. Such constitutive laws are not uncommon in the literature. In the present setting, Korn's inequality depends on the curvature of the surface $\Gamma_h$ due to the tangential jump of the velocities at the surface (see below \eqref{kin_bc}). Hence, challenges arise in the application of Korn's equation, due to the absence of sufficient regularity of $\Gamma_h$. This can be remedied if a more regular structure is considered.

Here $\rho_f$ is the constant fluid density which, for simplicity, will be set to 1.

We impose the following Navier-slip condition at the bottom boundary $\Gamma_b$:
\begin{equation}
    \begin{split}\label{bottom_bc}
       &   \bu\cdot \bn^b=0, \quad \text{and}    \quad  \bu\cdot\bftau^b=-\beta_b(\boldsymbol{\sigma} \bn^b)\cdot\bftau^b \text{ on }\Gamma_b,
    \end{split}
\end{equation}
where $\bn^b=(0,-1)$ is the unit normal to $\Gamma_b$ and $\bftau^b=(1,0)$. The proportionality constant $\beta_b>0$ is the slip-length.

At the inlet and the outlet of the channel, we assume that the flow is purely transversal. At these boundaries, we also impose the following dynamic pressure, and fixed outflow rate conditions: 
\begin{equation}
    \begin{split}\label{bc_inout}
        &  p(t)+\frac12{|u_1(t)|^2} =\Pinout(t),   \quad \text{ and }\quad u_2(t) = 0, \quad \text{ on } \Gamma_{in/out},
     \\   & { \int_{\Gamma_{out}}u_1(t) = 1},
    \end{split}
\end{equation}
where the given time-dependent pressure data $\Pinout$ are assumed to be in $L^2_{loc}(0,\infty)$. Moreover, we assume that the following global-in-time bound is satisfied at the inlet only: For some constant $C_0>0$, 
\begin{align}\label{pressureL^2}
    \|\Pin\|_{L^2(0,\infty)}\leq C_0.
\end{align}
To prove the desired finite-time contact of the top compliant boundary with the bottom boundary, we will also suppose, for some constant pressure $p_0>0$, that
\begin{align}
    \label{pressure_difference}
    \Pout(t)-\Pin(t) \geq p_0,\qquad \text{ for all times } t\geq 0.
\end{align}
We assume that the structure can be deformed only in the direction normal to the bottom boundary $\Gamma_b$.
The elastodynamics problem, describing the height of the thin elastic structure from the bottom boundary, is given by the following plate equations:
\begin{equation}
   \rho_s \partial_{tt}h +\alpha\partial_{xxxx}h-\gamma\partial_{xxt}h = f(\bu,p,h),\qquad \text{ in } (0,L).
\label{beam_eqn}
\end{equation}
Here $\alpha>0$ and $\gamma>0$ are the elastic and visco-elastic coefficients, respectively, $\rho_s$ is the density of the structure which will be set to 1, and $f$ is the load to the structure. Given that there are no external forces on the structure, this force in the coupled problem results from the jump in the normal stress (traction) across the structure. This is described below in \eqref{dyn_bc}.

The structure assumed to be clamped and satisfies the following boundary conditions:
\begin{equation}
    \begin{split}\label{bc_h}
           h(0,t) = h(L,t) = H,\\
        \partial_x h(0,t) = \partial_x h(L,t) =0.
    \end{split}
\end{equation}
The fluid and the structure are coupled via two coupling conditions and these interactions are described in detail below. 

1. The {\bf kinematic coupling condition} describes the behavior of kinematic entities at the fluid-structure interface. We impose the following Navier-slip coupling condition, which reads,
\begin{equation}\label{kin_bc}
    \begin{split}
       & (0,\partial_t h)\cdot \bn^h = \bu\cdot \bn^h, \text{ on } \Gamma_h\\
       & ((0,\partial_t h)-\bu)\cdot \bftau^h = \beta_s \boldsymbol{\sigma} \bn^h\cdot \bftau^h, \text{ on } \Gamma_h.
    \end{split}
\end{equation}
Here, 
\begin{align}\label{normaltangent}
    \bn^h=\frac{(-\partial_xh,1)}{\sqrt{1+|\partial_xh|^2}}, \qquad \text{and}\qquad  \bftau^h={ \frac{(1,\partial_xh)}{\sqrt{1+|\partial_xh|^2}}},
\end{align}
are, respectively, unit normal and tangential along the top time-dependent (unknown) boundary $\Gamma_h$ and $\beta_s>0$ is the slip-length. 

The first equation in \eqref{kin_bc} describes the continuity of fluid and structure velocities in the direction $\bn^h$ normal to $\Gamma_h$ thus implying that the fluid does not permeate through the top structure boundary. Unlike the no-slip coupling condition, wherein the fluid sticks to the top boundary and thus has no motion relative to the structure, the Navier-slip coupling condition allows for the motion of the fluid in the direction tangential to the boundary $\Gamma_h$. This is reflected in the second equation in \eqref{kin_bc} which implies that the jump in the fluid and structure velocities in the direction $\bftau_h$, tangential to $\Gamma_h$, is proportional to the tangential component of the normal stresses at the interface, where the proportionality constant is the structure slip-length $\beta_s$.\\

2. The {\bf dynamic coupling condition} describes the balance of forces at the fluid-structure interface and reads,
\begin{equation}\label{dyn_bc}
    f(\bu,p,h) = -\cal{S}_h(\bu,p)\boldsymbol{\sigma}(\bu,p)\bn^h \cdot (0,1), \quad \text{ on }\Gamma_h.
\end{equation}
Here, $\cal{S}_h=\sqrt{1+|\partial_xh|^2}$ is the surface measure of $\Gamma_h$, i.e. $d\Gamma_h=\cal{S}_hd\Gamma$. This condition states that the structure elastodynamics is driven by the jump in the normal stress across the interface $\Gamma_h$.

Finally, this system is supplemented with the following initial conditions:
\begin{align}\label{ic}
	\bu(t=0)=\bu_0,\quad {h}(t=0)={h}_0, \quad \partial_t{h}(t=0)=v_0.
\end{align}

\subsection{The energy of the problem}
In this section, we formally derive the energy of the coupled fluid-structure system.
We begin by multiplying the fluid equation by $\bu$ and integrating over the fluid domain ${F_h(t)}$, which yields
\begin{equation*}
    \int\limits_{{F_h(t)}} (\partial_t \bu) \cdot \bu 
    + \int\limits_{F_h(t)} (\bu\cdot \nabla \bu)\cdot \bu 
    - \int\limits_{{F_h(t)}} (\nabla\cdot\boldsymbol{\sigma}) \cdot\bu 
    =0
\end{equation*}
Using the Reynolds' transport theorem (see e.g. \cite{CKMT25}) with domain velocity $w$ on the first term yields
\begin{equation*}
    \int\limits_{{F_h(t)}} (\partial_t \bu)\cdot \bu 
    = \frac{1}{2}\frac{d}{dt}\int\limits_{{F_h(t)}} |\bu|^2 - \frac{1}{2} \int\limits_{ {\Gamma_h(t)}}  |\bu|^2 \bv \cdot \bn^h,
\end{equation*}
where we use the following notation to denote the velocity of the compliant boundary,
$$\bv := (0,\partial_t h).$$

For the second term, incompressibility gives us $\displaystyle(\bu\cdot \nabla \bu) \cdot \bu = \frac{1}{2}\nabla \cdot (|\bu|^2 \bu)$ and we write,
\begin{equation*}
    \int\limits_{{F_h(t)}}(\bu\cdot \nabla \bu)\cdot \bu 
    =\frac{1}{2}\int\limits_{\partial {F_h(t)}} |\bu|^2 (\bu\cdot \bn)=\frac{1}{2}\int\limits_{ {\Gamma_h(t)}} |\bu|^2 (\bv\cdot \bn^h)- \frac{1}{2}\int\limits_{ {\Gamma_{in}}} |\bu|^2 u_1+\frac{1}{2}\int\limits_{ {\Gamma_{out}}} |\bu|^2 u_1.
\end{equation*}
Next, we consider the third and final term.
Since $\nabla \cdot \bu =0$ we have 
$$ \int\limits_{{F_h(t)}}(\nabla\cdot\boldsymbol{\sigma}) \cdot\bu =- \mu \int\limits_{{F_h(t)}}|\nabla\bu|^2+\int\limits_{\partial {F_h(t)}}\boldsymbol{\sigma} \bn \cdot \bu,$$
where $\bn$ is the unit normal to the fluid domain boundary $\partial F_h$.
We split the boundary integral on the right-hand side of the equation above.
At the bottom boundary we apply the  Navier-slip condition \eqref{bottom_bc} to obtain that,
\begin{equation*}
    \int\limits_{\Gamma_b} \boldsymbol{\sigma} \bn^b \cdot \bu 
    = -\frac{1}{\beta_b} \int\limits_{\Gamma_b}|\bu\cdot \bftau^b|^2.
\end{equation*}
Next, we turn to the moving boundary and using the slip conditions \eqref{kin_bc}$_1$, we obtain
\begin{equation*}
    \int\limits_{\Gamma_h} \boldsymbol{\sigma} \bn^h\cdot \bu 
    = \int\limits_{\Gamma_h} \boldsymbol{\sigma} \bn^h \cdot (\bu-\bv) 
    + \int\limits_{\Gamma_h} \boldsymbol{\sigma} \bn^h \cdot \bv
    = -\frac{1}{\beta_s} \int\limits_{\Gamma_h} |(\bu-\bv)\cdot \bftau^h|^2 
    + \int\limits_{\Gamma_h} \boldsymbol{\sigma} \bn^h \cdot \bv.
\end{equation*}
Observe that, due to the dynamic pressure data at the inlet and the outlet boundaries \eqref{bc_inout}, we also have, 
\begin{equation*}
\begin{split}
    \int\limits_{\Gamma_{in}} \boldsymbol{\sigma} \bn \cdot \bu = \int\limits_{\Gamma_{in}} \Pin u_1 
    - \frac{1}{2}\int\limits_{ {\Gamma_{in}}} |\bu|^2 u_1,\qquad
    \int\limits_{\Gamma_{out}} \boldsymbol{\sigma} \bn \cdot \bu =
    -\int\limits_{\Gamma_{out}} \Pout u_1+\frac{1}{2}\int\limits_{ {\Gamma_{out}}} |\bu|^2 u_1.
\end{split}
\end{equation*}
Thus combining these identities we arrive at,
\begin{equation} \label{fluid_en}
   \begin{split}
        \frac{1}{2}\frac{d}{dt}\int\limits_{{F_h(t)}}  |\bu|^2
  &  +\mu \int\limits_{{F_h(t)}}|\nabla\bu|^2 + \frac{1}{\beta_b} \int\limits_{\Gamma_b}|\bu\cdot \bftau^b|^2
    +\frac{1}{\beta_s} \int\limits_{\Gamma_h} |(\bu-\bv)\cdot \bftau^h|^2 \\
    &
    =\int\limits_{\Gamma_{in}} \Pin u_1
    -\int\limits_{\Gamma_{out}} \Pout u_1
    + \int\limits_{\Gamma_h} \boldsymbol{\sigma}\bn^h \cdot \bv.
   \end{split}
\end{equation}
We next multiply the beam equation \eqref{beam_eqn} with the structure velocity $\partial_t h$, integrate and then integrate by parts. This yields,
\begin{equation} \label{beam_en}
    \frac{1}{2} \frac{d}{dt} \int\limits_0^L  |\partial_t h|^2 
    + \frac{1}{2} \frac{d}{dt}
     \int\limits_0^L\alpha |\partial_{xx}h|^2 
    + \gamma \int\limits_0^L |\partial_{xt} h|^2 
    =- \int\limits_{\Gamma_h} \boldsymbol{\sigma} \bn^h \cdot \bv
\end{equation}
Hence, by combining the two equations \eqref{fluid_en} and \eqref{beam_en} and integrating in time over $[0,t]$, we arrive at,
\begin{equation}
  \begin{split}\label{en_eqn}
        &\frac{1}{2} \left[ \int\limits_{{F_h(t)}}  |\bu(t)|^2
    +\int\limits_0^L  |\partial_t h(t)|^2 
    + \alpha \int\limits_0^L |\partial_{xx}h(t)|^2
    \right] 
    + \mu \int\limits_0^t \int\limits_{F_h(s)} |\nabla\bu|^2
    \\ 
    & + \frac{1}{\beta_b}\int\limits_0^t \int\limits_{\Gamma_b} |\bu\cdot \bftau^b|^2
    + \frac{1}{\beta_s}\int\limits_0^t \int\limits_{\Gamma_h} |(\bu-\bv)\cdot \bftau^h|^2 
    + \gamma \int\limits_0^t \int\limits_0^L |\partial_{xt}h|^2 
    \\  
    &= \frac{1}{2} \left[\int\limits_{F_{h_0}} |\bu_0|^2 
    + \int\limits_0^L  |v_0|^2
    + \alpha \int\limits_0^L |\partial_{xx}h_0|^2
    \right]
     + \int\limits_0^t \left( \int\limits_{\Gamma_{in}} \Pin u_1 - \int\limits_{\Gamma_{out}}\Pout u_1\right).
  \end{split}
\end{equation}
To complete this energy estimate, for any $T>0$, we next take $\sup_{t\in[0,T]}$ on both sides of the equation above and then find bounds for the boundary integrals appearing on the right-hand side of the equation above. Instead of using the typical trace theorem where the constant of continuity depends inversely on the height of the fluid domain, i.e. the function $h$ (see e.g. \cite{CDEG,CKMT25}), which is not amenable to our current setting of analyzing near-contact dynamics, we make the following observations: Firstly, due to the incompressibility condition, we have 

\begin{align*}
   \int_{\Gamma_{in}} \bu|_{\Gamma_{in}} \cdot (-1,0) +    \int_{\Gamma_{out}} \bu|_{\Gamma_{out}} \cdot (1,0) + 
   \int_{\Gamma_{h}} \bu|_{\Gamma_{h}} \cdot \bn^h = \int_{F_h}\nabla\cdot{\bu} =0.
\end{align*}
Using the continuity of fluid and structure velocities in the normal direction at the fluid-structure interface, namely the kinematic condition \eqref{kin_bc}, we further simplify this equation as follows,
\begin{align}\label{fluxdF}
 \int_{\Gamma_{out}} u_1   -\int_{\Gamma_{in}} u_1 = -\int_0^L  \partial_t h . 
\end{align}

Hence, the second term on the right hand side of \eqref{en_eqn} can be written as,
\begin{align*}
    \int_0^T\left(\Pin\int_{\Gamma_{in}}u_{1}- \Pout\int_{\Gamma_{out}}u_{1} \right)   
    = \int_0^T\left(\Pin\int_0^L  \partial_t h+ (\Pin-\Pout)\int_{\Gamma_{out}}u_1\right).
\end{align*}
Recall the boundary condition \eqref{kin_bc}$_2$ which imposes the fixed outflux condition $\displaystyle\int_{\Gamma_{out}}u_1=1$. Then, due to the assumption that $P_{in}(t)-P_{out}(t)\leq 0$ for all $t\geq0$, stated in \eqref{pressure_difference}, the second term on the right hand side of the equation above is negative. Hence, we obtain,
 $$ \int_0^T\left(\Pin\int_{\Gamma_{in}}u_{1}- \Pout\int_{\Gamma_{out}}u_{1} \right) \leq C_\gamma\|\Pin\|_{L^2(0,\infty)}^2 +\frac{\gamma}{2}\int_0^T\|\partial_{tx}h\|^2_{L^2(0,L)},$$
 where the constant $C_\gamma>0$ depends only on $\gamma$ and is independent of $T$.
 We stress that it is crucial for this constant to be independent of $T$. We also remark that, with these a-priori estimates at our disposal, we can derive an appropriate trace inequality (see \eqref{trace0}).

\subsection{Weak formulation on the moving domain}
Using the convention that bold-faced letters denote spaces containing vector-valued functions, we define the following relevant function spaces for the fluid velocity and the structure displacement motivated by the energy inequality obtained above:
\begin{align*}
  &  {\cal{V}}_{F_h(t)} = \{ \bu\in \bH^1({F_h(t)}):\,\,\nabla\cdot \bu=0, \,\, u_2\big|_{ \Gamma_b,\Gamma_{in},\Gamma_{out}}=0 \}, \\
  &  {\cal{W}}_F(0,T)=L^\infty(0,T;\bL^2({F_h})) \cap L^2(0,T;{\mathcal{V}}_{F_h}),\\
   & {\mathcal{W}}_S(0,T)=W^{1,\infty}(0,T;L^2(0,L))\cap L^\infty(0,T;H^2(0,L))\cap H^1(0,T;H^1_0(0,L)),\\
  &  {\mathcal{W}}(0,T)=\{(\bu,h)\!\in\!{\mathcal{W}}_F(0,T)\times {\mathcal{W}}_S(0,T):\\
  &\hspace{1in}\bu(t,x,h(x,t))\cdot \bn^h=(0,\partial_t h(x,t))\cdot \bn^h, (t,x)\in [0,T]\times[0,L] \},\\
  &  {\cal{T}}_F(0,T)=\{ \bq\in H^1(0,T;\bL^2({F_h})) \cap L^2(0,T;\bH^1({F_h})), \,\,\,\nabla\cdot\bq=0,\,\, q_2|_{ \Gamma_b,\Gamma_{in},\Gamma_{out}}=0\},\\
   & {\cal{T}}_S(0,T) = H^1(0,T;L^2(0,L)) \cap L^2(0,T;H^2(0,L)),\\
   & {\cal{T}}(0,T) = \{(\bq,\zeta)\in {\cal{T}}_F(0,T) \times {\cal{T}}_S(0,T): \\
    &\hspace{1in} \bq(t,x,h(x,t)) \cdot \bn^h = (0,\zeta(x,t) )\cdot \bn^h, (t,x) \in [0,T]\times[0,L]\}.
\end{align*}
We note here that the continuity of velocities in the normal direction at $\Gamma_h$ is part of our solution and test spaces. The jump in the tangential direction is enforced weakly.

In what follows, we aim to derive the weak formulation of the problem \eqref{fluid_eqn}-\eqref{dyn_bc}. We consider a test pair $(\bq,\zeta)\in \cal{T}(0,T)$. 
We begin by formally multiplying the fluid equations \eqref{fluid_eqn} with $\bq$, then integrating in time over $[0,T]$, for some $T>0$. This yields,
\begin{equation}\label{inter}
   \int\limits_0^T\int\limits_{{F_h(t)}} \partial_t \bu \cdot \bq 
   + \int\limits_0^T\int\limits_{{F_h(t)}} (\bu\cdot \nabla \bu) \cdot \bq
   - \int\limits_0^T\int\limits_{{F_h(t)}} (\nabla \cdot \boldsymbol{\sigma})\cdot \bq 
   = 0.
\end{equation}
We will treat each of the terms individually.

    \textbullet\, 
    With the aid of the Reynolds transport theorem (see e.g. \cite{CKMT25}) we write the first term on the left side of \eqref{inter} as,
    $$\int\limits_0^T\int\limits_{{F_h(t)}} \partial_t \bu \cdot \bq 
    = \left[\int\limits_{F_h(s)} \bu\cdot\bq\right]_{s=0}^{s=T}
    -\int\limits_0^T\int\limits_{{F_h(t)}}\bu\cdot\partial_t \bq
    - \int\limits_0^T\int\limits_{ {\Gamma_h(t)}} (\bv \cdot \bn^h)(\bu\cdot\bq),$$
     where we denote by
    \begin{align}\label{structurevelocity}
        \bv=(0,\partial_t h(x,t)),
    \end{align}
 the velocity of the top boundary $\Gamma_h$.

     \textbullet\,  For the advection term in \eqref{inter}, integration by parts yields,
  \begin{align*}
	\int\limits_0^T\int\limits_{{F_h(t)}} (\bu\cdot\nabla)\bu\cdot\bq&=\frac12\int\limits_0^T\int\limits_{{F_h(t)}} (\bu\cdot\nabla)\bu\cdot\bq-\frac12\int\limits_0^T\int\limits_{{F_h(t)}} (\bu\cdot\nabla)\bq\cdot\bu +\frac12\int\limits_0^T\int\limits_{\partial F_h(t)}(\bu\cdot\bq )(\bu\cdot \bn^{h}) \\
	&=\frac12\int\limits_0^T\int\limits_{{F_h(t)}} (\bu\cdot\nabla)\bu\cdot\bq-\frac12\int\limits_0^T\int\limits_{{F_h(t)}} (\bu\cdot\nabla)\bq\cdot\bu +\frac12\int\limits_0^T\int\limits_{\Gamma_h(t)}(\bu\cdot\bq)(\bu\cdot \bn^{h}) \\
    &- \frac12\int\limits_0^T\int_{\Gamma_{in}}|u_1|^2q_1+\frac12\int\limits_0^T\int_{\Gamma_{out}}|u_1|^2q_1.
\end{align*}
      
    \textbullet\,  For the last term in \eqref{inter}, we apply integration by parts and obtain, 
    $$\int\limits_0^T\int\limits_{{F_h(t)}} (\nabla \cdot \boldsymbol{\sigma})\cdot \bq 
    = -\int\limits_0^T\int\limits_{{F_h(t)}}\mu \nabla\bu : \nabla\bq 
    + \int\limits_0^T\int\limits_{\partial {F_h(t)}} \boldsymbol{\sigma} \bn \cdot \bq,$$
   where $\bn$ is the unit normal at the boundary $\partial F_h$.
    
We will next analyze the boundary term. For that purpose we recall the Navier-slip boundary conditions \eqref{bottom_bc}, \eqref{kin_bc} that are imposed at $\Gamma_b$ and $\Gamma_h$, respectively.
Denoting $\bfzeta=(0,\zeta)$, we find,
    \begin{align*}
    & \int\limits_0^T\int\limits_{\Gamma_b}\boldsymbol{\sigma} \bn^b\cdot\bq 
        + \int\limits_0^T\int\limits_{\Gamma_h(t)}\boldsymbol{\sigma} \bn^h \cdot\bq \\
        &= \int\limits_0^T\int\limits_{\Gamma_b}\boldsymbol{\sigma} \bn^b \cdot((\bq\cdot \bn^b)\bn^b + (\bq\cdot\bftau^b)\bftau^b) 
        + \int\limits_0^T\int\limits_{\Gamma_h(t)}\boldsymbol{\sigma} \bn^h\cdot((\bq\cdot \bn^h)\bn^h + (\bq\cdot\bftau^h)\bftau^h) \\
        &= -\frac{1}{\beta_b}\int\limits_0^T\int\limits_{\Gamma_b}(\bu\cdot\bftau^b)\cdot (\bq\cdot\bftau^b) 
        - \frac{1}{\beta_s}\int\limits_0^T\int\limits_{\Gamma_h(t)}((\bu-\bv)\cdot\bftau^h)\cdot (\bq\cdot\bftau^h)
+\int\limits_0^T\int\limits_{\Gamma_h(t)}\boldsymbol{\sigma} \bn^h\cdot \bn^h(\bfzeta\cdot \bn^h).
    \end{align*}
Next, we consider the inlet and the outlet boundaries. By denoting $\bq=(q_1,q_2)$ we observe that,
\begin{align*}
  \int\limits_0^T \int\limits_{\Gamma_{in/out}} \boldsymbol{\sigma} \bn \cdot \bq  =\int\limits_0^T \int\limits_{\Gamma_{in/out}}\pm p\, q_1 =\pm \int\limits_0^T \int\limits_{\Gamma_{in/out}} \left( \Pinout-\frac12|u_1|^2\right) q_1 .
\end{align*}
Now we plug all the above identities into \eqref{inter} and obtain

\begin{equation} \label{WF_fluid}
\begin{split}
      & \left[\int\limits_{F_h(t)} \bu(t)\cdot\bq(t)\right]_{t=0}^{t=T}
     -\int\limits_0^T\int\limits_{{F_h(t)}}\bu\cdot\partial_t \bq   
     + \int\limits_0^T\int\limits_{{F_h(t)}}\mu \nabla\bu : \nabla\bq  -\frac12\int\limits_0^T\int\limits_{\Gamma_h(t)}(\bu\cdot\bq)(\bu\cdot \bn^{h})\\
    & +\frac12\int\limits_0^T\int\limits_{{F_h(t)}} (\bu\cdot\nabla)\bu\cdot\bq -\frac12\int\limits_0^T\int\limits_{{F_h(t)}} (\bu\cdot\nabla)\bq\cdot\bu -\int\limits_0^T\int\limits_{\Gamma_h(t)}\boldsymbol{\sigma} \bn^h\cdot \bn^h(\bfzeta\cdot \bn^h)\\ 
 &+\frac{1}{\beta_b}\int\limits_0^T\int\limits_{\Gamma_b}(\bu\cdot\bftau^b)\cdot (\bq\cdot\bftau^b) 
        + \frac{1}{\beta_s}\int\limits_0^T\int\limits_{\Gamma_h(t)}((\bu-\bv)\cdot\bftau^h)\cdot (\bq\cdot\bftau^h)
= \int\limits_0^T \int\limits_{\Gamma_{in}} \Pin q_1 
	-\int\limits_0^T \int\limits_{\Gamma_{out}}  \Pout q_1.
\end{split}
\end{equation}

Next, we multiply the structure equation \eqref{beam_eqn} with the test function $\zeta$,
and integrate by parts the terms on the left side. Then by applying the dynamic coupling conditions \eqref{dyn_bc} to the term on the right of the structure equation \eqref{beam_eqn}, 
and noticing for $\bfzeta=(0,\zeta)$, that
\begin{align*}
 \int\limits_0^T\int\limits_0^L\cal{S}_h\boldsymbol{\sigma} \bn^h \cdot \bfzeta ds
    & = \int\limits_0^T\int\limits_{\Gamma_h(t)}\boldsymbol{\sigma} \bn^h\cdot \bn^h(\bfzeta\cdot \bn^h) 
    + \frac{1}{\beta_s}\int\limits_0^T\int\limits_{\Gamma_h(t)}((\bv-\bu)\cdot\bftau^h) \cdot (\bfzeta\cdot\bftau^h),
\end{align*}
we obtain, 
\begin{align} \label{beam_int}
    \left[\int\limits_{0}^{L}\partial_th(t)\zeta(t) \right]_{t=0}^{t=T}   - \int\limits_0^T\int\limits_{0}^{L}\partial_th\partial_t\zeta 
    + \int\limits_0^T\int\limits_{0}^{L}\alpha \partial_{xx}h\partial_{xx}\zeta 
    + \int\limits_0^T\int\limits_{0}^{L}\gamma\partial_{xt}h\partial_x\zeta\\
    \nonumber
    +\int\limits_0^T\int\limits_{\Gamma_h(t)}\boldsymbol{\sigma} \bn^h\cdot \bn^h(\bfzeta\cdot \bn^h)
    +\frac{1}{\beta_s}\int\limits_0^T\int\limits_{\Gamma_h(t)}((\bv-\bu)\cdot\bftau^h) \cdot (\bfzeta\cdot\bftau^h)=0.
\end{align}

Thus the weak formulation of our problem is attained by adding together the identities \eqref{WF_fluid} and \eqref{beam_int}:
	For any given $T>0$, we seek $(\bu,h) \in {\mathcal{W}}(0,T)$, with $\int_{\Gamma_{out}}u_1=1$ that satisfies the following equation for any test function $(\bq,\zeta) \in  {\mathcal{T}}(0,T)$ such that $\int_{\Gamma_{out}}q_1=1$,
 \begin{equation} \label{weak_form_gen}
  \begin{split}
     &  \int\limits_{F_h(T)} \bu(T)\cdot\bq(T)+\int\limits_{0}^{L}\partial_th(T)\zeta(T)  
     -\int\limits_0^T\int\limits_{{F_h(t)}}\bu\cdot\partial_t \bq   
     + \int\limits_0^T\int\limits_{{F_h(t)}}\mu \nabla\bu : \nabla\bq-\frac12\int\limits_0^T\int\limits_{\Gamma_h(t)}(\bu\cdot\bq)(\bu\cdot \bn^{h})
     \\
   & +\frac12\int\limits_0^T\int\limits_{{F_h(t)}} (\bu\cdot\nabla)\bu\cdot\bq -\frac12\int\limits_0^T\int\limits_{{F_h(t)}} (\bu\cdot\nabla)\bq\cdot\bu +\frac{1}{\beta_b}\int\limits_0^T\int\limits_{\Gamma_b}(\bu\cdot\bftau^b)\cdot (\bq\cdot\bftau^b)  \\
  & 
        + \frac{1}{\beta_s}\int\limits_0^T\int\limits_{\Gamma_h(t)}((\bu-\bv)\cdot\bftau^h)\cdot ((\bq-\boldsymbol{\zeta})\cdot\bftau^h)
         - \int\limits_0^T\int\limits_{0}^{L}\partial_th\partial_t\zeta 
  + \int\limits_0^T\int\limits_{0}^{L}\alpha \partial_{xx}h\partial_{xx}\zeta 
    + \int\limits_0^T\int\limits_{0}^{L}\gamma\partial_{xt}h\partial_x\zeta \\
&= \int\limits_{F_{h_0}} \bu_0\cdot\bq(0)+ \int\limits_{0}^{L}v_0\zeta(0)  
+\int\limits_0^T \int\limits_{\Gamma_{in}} \Pin q_1 
	-\int\limits_0^T \int\limits_{\Gamma_{out}}  \Pout q_1,
  \end{split}
\end{equation}
where we use the notation $\bv=(0,\partial_t h)$.

We are now in a position to state our main results. This will be done in the following section.

\section{Main results}\label{sec:mainresults}
In this section we present the main results of this paper. 
First, we recall here the results given in \cite{MC16, T24} that provide a proof of existence of weak solutions  locally-in-time. A modification of the approaches used in these papers, give us the following existence result that incorporates the new boundary condition \eqref{bc_inout}$_2$.
\begin{theorem}\label{thm:exist}
	Let the initial data for structure displacement, structure velocity and fluid velocity be such that $h_0-H \in H^{2}_0(0,L),v_0\in H^1_0(0,L)$ and $\bu_0\in \bL^2(F_{h_0})$. Assume that the dynamic pressure data at the inlet and the outlet boundaries are $P_{in}\in L^2(0,\infty), P_{out}\in L^2_{loc}(0,\infty)$. 
     Then there exists $T_0>0$ and at least one weak solution of the system \eqref{fluid_eqn}-\eqref{dyn_bc} given by $(\bu,h)\in \cal{W}(0,T_0)$  satisfying the outlet flux condition,
   \begin{align}\label{uoutlet}
       \int_{\Gamma_{out}}u_1 
       =1,\qquad\text{ for almost every }t< T_0 ,
   \end{align}
    and satisfying the weak formulation \eqref{weak_form_gen} on $[0,T_0)$ for any test function $(\bq,\zeta)\in \cal{T}(0,T)$ with,
\begin{align}\label{phioutlet}
 \int_{\Gamma_{out}}q_1
=1, \qquad\text{ for almost every }t< T_0 .
\end{align}

    This weak solution further satisfies the following energy inequality for some constant $C>0$ independent of $T_0$:
    \begin{align}\label{en_ineq}
    \sup\limits_{t\in[0,T_0]} & 
    \left( 
    \|\bu(t)\|^2_{L^2({F_h(t)})} 
    + \| \partial_t h (t)\|^2_{L^2(0,L)} 
    + \|\partial_{xx}h(t)\|^2_{L^2(0,L)}
    \right)
    \\ \nonumber
    & + \int\limits_0^{T_0} \int\limits_{{F_h(t)}}|\nabla\bu|^2
    + \frac{1}{\beta_b}\int\limits_0^{T_0} \int\limits_{\Gamma_b} |\bu\cdot\bftau^b|^2
    + \frac{1}{\beta_s}\int\limits_0^{T_0} \int\limits_{\Gamma_h} |(\bu-(0,\partial_th))\cdot\bftau^h|^2 
    + \gamma \int\limits_0^{T_0} \int\limits_0^L |\partial_{xt}h|^2 
    \\ \nonumber &
    \leq \int\limits_{F_{h_0}}\rho_f |\bu_0|^2 
    + \int\limits_0^L \rho_s |v_0|^2
   + \int\limits_0^L |\partial_{xx} h_0|^2  + C\int\limits_0^\infty \left( P_{in}  \right)^2dt\\
   &:=K_0.\nonumber
\end{align}

    Moreover, the time $T_0>0$ is the first instance of contact of the elastic structure with the bottom boundary. In other words, $h(x_0,T_0)=0$ for some $x_0\in(0,L)$.
\end{theorem}
Before moving on, we list below some important consequences of Theorem \ref{thm:exist}, specifically those of the energy inequality \eqref{en_ineq}:
   \begin{itemize} 
 \item We emphasize that the constant $K_0>0$ in \eqref{en_ineq} and hence all the forthcoming constants, do not depend on $T$.
   \item Note that due to the Sobolev embedding $H^2(0,L) \hookrightarrow C^{1,\frac12}([0,L])$ in one-dimension, any structure displacement satisfying the bounds \eqref{en_ineq} also satisfies
\begin{align}\label{hbounded}
 \|h\|_{L^\infty(0,T;C^{1,\frac12}([0,L]))} \leq \tilde C,
\end{align}
where the constant $\tilde C>0$ only depends on $L$ and $K_0$, and is independent of $T$.
\item Note that, since 
$$\bu|_{\Gamma_h} = (\bu|_{\Gamma_h}\cdot\bftau^h)\bftau^h + (\bu|_{\Gamma_h}\cdot \bn^h)\bn^h= ((\bu|_{\Gamma_h}-(0,\partial_th))\cdot\bftau^h)\bftau^h + (0,\partial_th),$$ we have, due to \eqref{en_ineq}, that
\begin{align}\label{uboundary}
    \|\bu\|_{L^2(0,T; \bL^2(\Gamma_h))} \leq     \|(\bu-(0,\partial_th))\cdot\bftau^h\|_{L^2(0,T; L^2(\Gamma_h))} + \|\partial_th\|_{L^2(0,T; L^2(0,L))} \leq C,
\end{align}
where $C>0$ depends only on the given data and is independent of $T$.
Here, we have additionally used the boundedness of the normal and the tangential defined in \eqref{normaltangent} i.e. the fact that for any $t\in [0,T], x\in[0,L]$ we have
\begin{align}\label{normtangbound}
    |\bftau^h|\leq C\qquad\text{and}\qquad|\bn^h| \leq C,
\end{align}
which is follows from \eqref{hbounded}.\\
The same can be deduced about the trace of $\bu$ on the bottom boundary $\Gamma_b$.
We remark here the importance of these bounds by noting that, the constant $C_{F_h}$ in the trace inequality $ \displaystyle  \|\bu\|_{L^2(0,T; \bL^2(\Gamma_h))} \leq C_{F_h}   \|\bu\|_{L^2(0,T; \bH^1(F_h))} $ depends inversely on the height $h$ of the tube. The lack of bounds on this constant in the near-contact case, is why using only the energy inequality \eqref{en_ineq} is necessary.
\end{itemize}

We will next state our first main result concerning the spatial regularity of the structure displacement.
\begin{theorem}\label{thm:regularity}
 Consider the weak solution $(\bu,h)$ obtained in Theorem \ref{thm:exist}. Then, until the first instance of contact $T_0>0$, the structure displacement satisfies the following additional spatial regularity:
        \begin{align} \label{reg_spatial}
        h-H\in {L^2(0,T;H_0^3(0,L))},\qquad\text{ for every } \,T<T_0.
        \end{align}
\end{theorem}
We will now state our second main result which proves the collapse of the compliant tube wall in finite time.
\begin{theorem}\label{mainresult}
Assume that the dynamic pressure data at the inlet and outlet satisfy $P_{in}\in L^2_{}(0,\infty)$ (i.e. \eqref{pressureL^2}) and $P_{out}\in L^2_{loc}(0,\infty)$.
We also suppose that \eqref{pressure_difference} holds, that is: For some $p_0>0$, the given dynamic inlet and outlet pressures satisfy,
$$
    P_{out}(t)-P_{in}(t) \geq p_0,\qquad \text{ for all }t\geq 0.
$$
Assume that the initial data satisfies $h_0-H\in H_0^2(0,L)$, $\bu_0\in \bL^2(F_{h_0})$ and $v_0\in H^1_0(0,L)$,
   then for the weak solution $(\bu,h)$ constructed in Theorem \ref{thm:exist}, there exists $0<T_0<\infty$, and $x_0\in (0,L)$ such that $h(x_0,T_0)=0$.
\end{theorem}

In what follows, we will prove these three theorems.
Firstly, in Section \ref{sec:exist}, we will give a brief proof of Theorem \ref{thm:exist}. This existence proof describes how the boundary condition \eqref{bc_inout}$_2$ is enforced. Then we will prove the extra spatial regularity result Theorem \ref{thm:regularity} in Section \ref{sec:regularity}. This result is crucial in the construction of an appropriate test function in the proof of Theorem \ref{mainresult} undertaken in Section \ref{sec:contact}.

\section{Existence result: Proof of Theorem \ref{thm:exist}}\label{sec:exist}

In this section, we present a brief proof of Theorem \ref{thm:exist} outlining all the key steps. The strategy we employ relies on\\
\indent (1) first proving the existence of a local-in-time weak solution,\\
\indent (2) and then extending this solution up to the time of first contact $T_0>0$.\\
We begin with step (1). While several components of our proof build on methods established in prior work, we explain here how the outlet boundary flow condition \eqref{uoutlet}, which has not been previously studied, is enforced.

To incorporate this restriction, we introduce a penalized weak formulation that approximates the weak form of the original system \eqref{fluid_eqn}-\eqref{dyn_bc}. 
That is, for any $\epsilon>0$, we augment the weak form \eqref{weak_form_gen} with a penalty term enforcing the flow rate condition \eqref{uoutlet} in the limit $\epsilon\to 0$. More precisely, for any $T>0$, we seek a solution $(\bu_\epsilon,h_\epsilon)$ satisfying the following weak form for any $(\bq,\zeta)\in\cal{T}(0,T)$:
 \begin{multline} \label{weak_form_pen}
   \left[\int\limits_{F_{h_\epsilon}(t)} \bu_{\epsilon}(t)\cdot\bq(t)+\int\limits_{0}^{L}\partial_th_\epsilon(t)\zeta(t)  \right]_{t=0}^{t=T} 
     -\int\limits_0^T\int\limits_{{F_{h_\epsilon}(t)}}\bu_{\epsilon}\cdot\partial_t \bq   
     + \int\limits_0^T\int\limits_{{F_{h_\epsilon}(t)}}\mu \nabla\bu_{\epsilon} : \nabla\bq
     \\
    +\frac12\int\limits_0^T\int\limits_{{F_{h_\epsilon}(t)}} (\bu_{\epsilon}\cdot\nabla)\bu_{\epsilon}\cdot\bq -\frac12\int\limits_0^T\int\limits_{{F_{h_\epsilon}(t)}} (\bu_{\epsilon}\cdot\nabla)\bq\cdot\bu_{\epsilon} -\frac12\int\limits_0^T\int\limits_{\Gamma_{h_\epsilon}(t)}(\bu_{\epsilon}\cdot\bq)(\bu_{\epsilon}\cdot \bn^h_{\epsilon}) \\
   +\frac{1}{\beta_b}\int\limits_0^T\int\limits_{\Gamma_b}(\bu_{\epsilon}\cdot\bftau_{\epsilon}^b)\cdot (\bq\cdot\bftau_{\epsilon}^b) 
        + \frac{1}{\beta_s}\int\limits_0^T\int\limits_{\Gamma_{h_\epsilon}(t)}((\bu_{\epsilon}-\bv_{\epsilon})\cdot\bftau_{\epsilon}^h)\cdot ((\bq-\boldsymbol{\zeta})\cdot\bftau_{\epsilon}^h)+\frac1{\epsilon}\int_0^T\left(\int_{\Gamma_{out}}(u_\epsilon)_1-1\right)\left(\int_{\Gamma_{out}}q_1-1\right)\\
         - \int\limits_0^T\int\limits_{0}^{L}\partial_t h_{\epsilon}\partial_t\zeta 
    + \int\limits_0^T\int\limits_{0}^{L}\alpha \partial_{xx}h_{\epsilon}\partial_{xx}\zeta 
    + \int\limits_0^T\int\limits_{0}^{L}\gamma\partial_{xt}h_\epsilon\partial_x\zeta 
= \int\limits_0^T \int\limits_{\Gamma_{in}} \Pin q_1 
	-\int\limits_0^T \int\limits_{\Gamma_{out}}  \Pout q_1.
    \end{multline}
For any fixed $\epsilon>0$, the existence of a weak solution $(\bu_{\epsilon},h_\epsilon)$, satisfying \eqref{weak_form_pen} locally-in-time weak is obtained by semi-discretizing the equation in time and applying a Lie-Trotter splitting scheme in the spirit of \cite{MC16, T24}. 
Following the steps in \cite{MC16, T24}, we can see that this solution satisfies the following energy inequality until the time of first contact $T^\epsilon_0$ which is proved to be strictly positive: For any $T<T^\epsilon_0$, we have
\begin{equation}
  \begin{split}\label{en_ineq_pen}
        &\sup_{t\in[0,T]} \left[ \int\limits_{{F_{h_\epsilon}(t)}}  |\bu_{\epsilon}(t)|^2
    +\int\limits_0^L  |\partial_t h_\epsilon(t)|^2 
    + \alpha \int\limits_0^L |\partial_{xx}h_{\epsilon}(t)|^2
    \right] 
    + \mu \int\limits_0^T \int\limits_{F(s)} |\nabla\bu_{\epsilon}|^2
    \\ 
    & + \frac{1}{\beta_b}\int\limits_0^T \int\limits_{\Gamma_b} |\bu_{\epsilon}\cdot \bftau_{\epsilon}^b|^2
    + \frac{1}{\beta_s}\int\limits_0^T \int\limits_{\Gamma_{h_\epsilon}} |(\bu_{\epsilon}-\bv_{\epsilon})\cdot \bftau_{\epsilon}^h|^2 +\frac{1}{2\epsilon}\int_0^T\left(\int_{\Gamma_{out}}(u_\epsilon)_1-1\right)^2
    + \frac{\gamma}2 \int\limits_0^T \int\limits_0^L |\partial_{xt}h_\epsilon|^2 
    \\  
    &\leq \frac{1}{2} \left[\int\limits_{F_{h_0}} |\bu_0|^2 
    + \int\limits_0^L  |v_0|^2
    + \alpha \int\limits_0^L |\partial_{xx}h_0|^2
    \right]
+  C {\epsilon}{}  (\|\Pin\|_{L^2(0,T)}^2+\|\Pout\|_{L^2(0,T)}^2) + \frac{C}{\gamma}\|P_{in}\|^2_{L^2(0,T)},
  \end{split}
\end{equation}
Recall our pressure data assumption that states that
$$\|\Pin\|_{L^2(0,T)} <C_0, \quad \|\Pout\|_{L^2(0,T)}< C_T,$$
for some $C_T>0$ depending on $T$.

In deriving this energy inequality, the inlet-outlet boundary terms appearing on the right hand side of \eqref{weak_form_pen} are handled as follows.
As earlier, the incompressibily condition implies that $\bu_{\epsilon}$ satisfies \eqref{fluxdF}, which then implies that,
\begin{align*}
    \int_0^T\left(\Pin\int_{\Gamma_{in}}(u_\epsilon)_1- \Pout\int_{\Gamma_{out}}(u_\epsilon)_1 \right)   
    \leq \int_0^T\left(\Pin\int_0^L  \partial_t h_\epsilon+ (\Pin-\Pout)\left(\int_{\Gamma_{out}}(u_\epsilon)_1-1\right)-p_0\right)
\end{align*}
Since $p_0>0$, we thus obtain, for some $C>0$ independent of $\epsilon$ that,
\begin{align*}
 \int_0^T&\left(\Pin\int_{\Gamma_{in}}(u_\epsilon)_1- \Pout\int_{\Gamma_{out}}(u_\epsilon)_1 \right) \leq \frac{C}{\gamma} \|\Pin\|_{L^2(0,T)}^2 +\frac{\gamma}{2}\int_0^T\|\partial_{tx}h_\epsilon\|^2_{L^2(0,L)}\\
    &\qquad\qquad+C {\epsilon}  \left(\|\Pin\|_{L^2(0,T)}^2+\|\Pout\|_{L^2(0,T)}^2\right) +\frac{1}{2\epsilon}\int_0^T\left(\int_{\Gamma_{out}}(u_\epsilon)_1-1\right)^2.
\end{align*}
The second and fourth terms on the right hand side of this inequality are absorbed by the left hand side, which gives us our desired energy inequality \eqref{en_ineq_pen} for the penalized problem.

We observe that one consequence of the energy \eqref{en_ineq_pen} is that,
\begin{align}
   \int_0^T\left(\int_{\Gamma_{out}}(u_\epsilon)_1-1\right)^2 \leq C\epsilon,
\end{align}
for some constant $C>0$ independent of $\epsilon$. This bound tells us that we recover the desired boundary behavior \eqref{uoutlet} in the limit $\epsilon\to0$.

Hence, at this stage our main goal is to pass $\epsilon\to0$ which, in the limit, will give us a local-in-time weak solution with the desired properties stated in Theorem \ref{thm:exist}.
However, this is not straightforward as the addition of the penalty term singularly perturbs the original problem \eqref{weak_form_gen}.
In particular, we emphasize that the penalty term is a hurdle in obtaining compactness of the approximate solutions. 
Therefore, in the following lemma we show how compactness result is obtained by showing that the fractional time derivatives of the fluid and structure velocities are bounded {\it independently} of $\epsilon$. As mentioned above, we first prove the existence of a local-in-time weak solution to the original problem \eqref{weak_form_gen} and then extend the solution in time until there is collapse of the elastic tube. Hence, for the purposes of passing $\epsilon\to0$, it suffices to analyze the approximate solutions $(\bu_\epsilon,h_\epsilon)$ only locally-in-time.

To that end, we consider the time $0<T^\delta_\epsilon\leq T$ that is decided by the following condition,
\begin{align}\label{hmin1}
  0<  \delta\leq h_\epsilon(x,t),\qquad \forall x\in[0,L],\,\, t\in [0,T^\delta_\epsilon],
\end{align}
ensuring a lower bound for the height of the fluid domain. 

Thanks to this lower bounds on the height of the fluid domain, it is convenient to establish the following result in a fixed reference domain. Hence, we introduce the Arbitrary Lagrangian-Eulerian (ALE) maps $\bA_{h_\epsilon}:\sO \to F_{h_\epsilon}$ that define a diffeomorphism between the fluid domain and the fixed domain $\sO=(0,L)\times(0,1)$ until time $T^\delta_\epsilon$ (see \eqref{hmin1}). These maps are given by the solution to:
\begin{equation}\label{ALE}
  \begin{split}
       \Delta \bA_{h_\epsilon}&= 0 \quad\text{ in }\sO,\\
    \bA_{h_\epsilon} = \Gamma_{h_\epsilon} \quad&\text{ on } \Gamma, \qquad \bA_{h_\epsilon} = {\bf Id}\quad\text{ on }\partial\sO\setminus\Gamma,
  \end{split}
\end{equation}
where $\Gamma=\{(x,1):x\in(0,L)\}$.

The reason behind for choosing the above harmonic extension is because of the added spatial regularity enjoyed by these ALE maps at any time:
$$\|\bA_{h_\epsilon}\|_{\bH^{2.5}(\sO)} \leq C\|h_\epsilon\|_{H^2(0,L)}$$

We now state our temporal regularity result.
\begin{lemma}\label{lem:compact}
Let $(\bu_{\epsilon},h_\epsilon)$ be a solution to \eqref{weak_form_pen} until some time $T^\epsilon_0>0$. Then, the fluid and structure velocities $(\bu_{\epsilon},\partial_t h_{\epsilon})$ satisfy the following Nikolski-space\footnote{
For any $0<m<1$ and $1\leq p<\infty$, the Nikolski space is defined as:
\begin{equation}\label{Nikolski}
	{N^{m,p}(0,T;\mathcal{Y})}=\{\bu\in L^p(0,T;\mathcal{Y}):\sup_{0<\kappa< T} \frac1{\kappa^m} \|\tau_\kappa\bu-\bu\|_{L^p(\kappa,T;\mathcal{Y})}<\infty\},
\end{equation}
where $\tau_\kappa f(t)=f(t-\kappa)$.
} 
bounds, for some constant $C_\delta>0$, depending on $\delta$ but independent of $\epsilon$, 
        	\begin{align}\label{reg_temp}
			\|\bu_{\epsilon}\circ \bA_{h_\epsilon}\|_ {N^{\frac18,2}(0,T^\delta_\epsilon;\bL^2((0,L)\times(0,1)))} + \|\partial_t h_{\epsilon}\|_{N^{\frac18,2}(0,T^\delta_\epsilon;L^2(0,L))}< C_\delta,
		\end{align}
        where the time $T^\delta_\epsilon>0$ is as defined in \eqref{hmin1}.
\end{lemma}
\begin{proof}
For the ease of readability, we temporarily set the height at inlet and outlet of the fluid domain $H=1$.

To introduce the Lagrangian formulation of the fluid problem we define,
$$\tilde\bu_{\epsilon}:=\bu_{\epsilon}\circ \bA_{h_\epsilon}.$$
For brevity, we denote 
$\nabla^{h}f=\nabla f(\nabla \bA_{h_\epsilon})^{-1},\text{ and } \nabla^{h}\cdot f=tr(\nabla^{h}f)$ and denote by $J_{h_\epsilon}$ the determinant of the map $\bA_{h_\epsilon}$.

Now, we consider the Lagrangian form of the weak formulation \eqref{weak_form_pen} given on the fixed domain $\sO$, which is obtained for $(\tilde\bu_\epsilon,h_\epsilon)$ by applying the change of variables $\bA_{h_\epsilon}$ defined in \eqref{ALE}. 

We obtain, for any $t$ for which $h_\epsilon(x,t) >0$, that the following equation holds for every $\bq\in H^1(0,T;\bL^2(\sO))\cap L^2(0,T;\bH^1(\sO))$ with $\nabla^h\cdot \bq=0$ and $\zeta \in \cal{T}_S(0,T)$ such that $\bq|_{\Gamma}\cdot\bn^h=(0,\zeta)\cdot\bn^h$:
	\begin{equation}\label{weak_form_fixed}
			\begin{split}
				&\left[{\int_{\sO}J_{h_\epsilon}(t)\tilde{\bu_{\epsilon}}(t)\bq(t) +\int_0^L\partial_t{h}_\epsilon(t)\zeta(t)}\right]_{t=0}^{t={T^\delta_*}} \\
                &
				=\int_0^{{T^\delta_*} }\int_{\sO}J_{h_\epsilon}\,\tilde{\bu_{\epsilon}}\cdot \partial_t\bq -\int_0^{T^\delta_*}\int_{\sO}J_{h_\epsilon}(\tilde{\bu_{\epsilon}}\cdot\nabla^{h }
		\tilde\bu_{\epsilon} \cdot\bq
		- \partial_t\bA_{h_\epsilon}\cdot\nabla^{h }\tilde{\bu_{\epsilon}}\cdot\bq )
		+\int_0^{T^\delta_*}\int_{\sO} 		
		\partial_tJ_{h_\epsilon}\tilde{\bu_{\epsilon}}\cdot\bq\\
        &- \nu\int_0^{{T^\delta_*} } \int_{\sO} J_{h_\epsilon}\nabla^h\tilde{\bu_{\epsilon}}\cdot \nabla^h\bq -\frac1{\alpha}\int_0^{T^\delta_*}\int_\Gamma \mathcal{S}_{h}((0,\partial_t{h}_\epsilon)-\tilde{\bu_{\epsilon}})\cdot\bftau_{\epsilon}^{h}((\bq-(0,\zeta))\cdot\bftau_{\epsilon}^{h})\\
			    &+ \int\limits_0^{T^\delta_*}\int\limits_{0}^{L}\partial_t h_{\epsilon}\partial_t\zeta 
    - \int\limits_0^{T^\delta_*}\int\limits_{0}^{L}\alpha \partial_{xx}h_{\epsilon}\partial_{xx}\zeta 
    -\int\limits_0^{T^\delta_*}\int\limits_{0}^{L}\gamma\partial_{xt}h_\epsilon\partial_x\zeta\\
    &-\frac1{\epsilon}\int_0^{T^\delta_*}\left(\int_{\Gamma_{out}} (\tilde u_\epsilon)_1-1\right)\left(\int_{\Gamma_{out}}q_1-1\right) 
				+\int_0^{{T^\delta_*} }\left( P_{{in}}\int_{\Gamma_{in}}q_1-P_{{out}}\int_{\Gamma_{out}}q_1\right).
		\end{split}\end{equation}
As earlier, we set $\bv_{\epsilon}= (0,v_\epsilon)=(0,\partial_t h_{\epsilon})$.

Our aim now is to pick appropriate test function for \eqref{weak_form_fixed} that will lead us to the semi-norm in \eqref{Nikolski}.
We fix $\kappa>0$ (see Definition \eqref{Nikolski}). Our method focuses on constructing test functions that take the following form:
\begin{align}\label{TFform}
    \bq(t)=\int_{t-\kappa}^t(\tilde\bu_{\epsilon})_{\text{mod}}(s,t)ds +k,\qquad \zeta(t)=\int_{t-\kappa}^t(v_\epsilon)_{\text{mod}}(s,t)ds+k,
\end{align}
where $(\tilde\bu_{\epsilon})_{\text{mod}}$ and $(v_\epsilon)_{\text{mod}}$ are appropriate modifications of $\tilde\bu_{\epsilon}$ and $v_\epsilon$, respectively and $k$ is a constant. This form of test functions leads to the time-translation term appearing in \eqref{Nikolski}. This motivation is further explained below in \eqref{lowerbound}.

The construction, in spirit, is similar to that in \cite{T24} but differs due to the outlet penalty term we have in \eqref{weak_form_fixed}. See also \cite{CMT26} in which a similar temporal regularity result is proved in the no-slip case, and \cite{G08} where the idea of using time integrals of appropriate modifications of the solution to construct the test function, was first developed for FSI. We only include certain key parts of the proof here-- We will give the definition the test functions, justify the motivation explained above, and then explain in particular how the penalty term is treated. The rest of the calculations are identical to those in \cite{T24, CMT26} and are thus skipped.

 Let ${\Pi_M}$ denote the orthonormal projector in $L^2(0,L)$ onto the space $\text{span}_{1\leq i\leq M}\{\varphi_i\}$, where $\varphi_i$ satisfies $-\Delta \varphi_i=\lambda_i\varphi_i$ and $\varphi_i=0$ on $x=0,x=L$. For any ${f}\in L^2(0,L)$ we notate $({f})_M={\Pi_M}{f}$.
Since we know that $\lambda_M \sim M^2$, we will choose
\begin{align*}
	\lambda_M = c\kappa^{-\frac34}.
\end{align*}

Then, the fluid test function is defined as follows
\begin{equation}
	\begin{split}\label{test_f}
		\bq(t)&:= \int_{t-\kappa}^t \left[ P_t^{-1} P_s (\tilde\bu_{\epsilon}(s)-\bw_\epsilon(s))	+\left( \bw^M_\epsilon(s)+ \mathsf{c}_\epsilon^M(s) \bfxi_0\chi \right) \right] ds \\
		&	-\int_{t-\kappa}^tP_t^{-1}\cal{B}\Big( \text{div} \left( P_s\bw_\epsilon(s)-P_t \bw^M_\epsilon(s)	-\mathsf{c}_\epsilon^M(s) P_t\bfxi_0\chi \right)\Big)ds -(\kappa-1)\\
		&:= \int_{t-\kappa}^t(\tilde\bu_{\epsilon})_{\text{mod}}(s,t)ds -(\kappa-1),
	\end{split}
\end{equation}
where,
\begin{itemize}
	\item	$\chi$ is a smooth function such that $\chi(x,1)= 1$ and $\chi(x,0)= 0$ for $x\in [0,L]$,
	\item $\bw_\epsilon=\bv_{\epsilon}\chi$ and $\bw^M_\epsilon=(\bv_{\epsilon})_M\chi$, are the extensions into the fluid domain of the structure velocity and its finite-dimensional projection, respectively.
\item $P_t({\bf f}):=(J_{h_\epsilon}\nabla \bA^{-1}_{h_\epsilon}\cdot {\bf f})(t)$ is the Piola transformation composed with the ALE map $\bA_{h_\epsilon}$ for any ${\bf f}\in\bL^2(\sO)$.
	\item $\sB$ is the Bogovski operator on the fixed domain $\sO$. Recall that (see e.g. \cite{ Galdi}) if $\int_\sO f=0$ then
	$$\nabla\cdot\sB(f)=f, \quad \text{and}\quad \|\sB(f)\|_{\bH^{1}_0(\sO)} \leq c\|{ f}\|_{L^2(\sO)}.$$
	\item The correction term $\mathsf{c}_\epsilon^M\bfxi_0\chi$ where $\mathsf{c}_\epsilon^M(s) 
    =\int_0^L(v_\epsilon(s)-(v_\epsilon)_M(s))\in L^\infty([0,T])$ is added to ensure that the condition $\int_\sO f=0$ motivated by the previous bullet point is satisfied. Note that, unlike in the no-slip case, these integral terms are not necessarily zero due to the nonzero fluid flux at the inlet and outlet.
	\item {$\bfxi_0 = (0,1)$. This choice ensures that $ \int_0^L (\bfxi_0 \cdot (-\partial_xh(t),1))= 1$ for any $t \in [0,T_\delta]$.} 
\end{itemize}
Next, we define the corresponding structure test function as follows:
\begin{align}\label{test_s}
	\zeta(t):=\int_{t-\kappa}^t\left( (v_\epsilon)_M(s)+
	\mathsf{c}_\epsilon^M(s)\right) ds  -(\kappa-1).
\end{align}
We remark here that, in the definition of $\bq$ and $\zeta$ above, we have extended $\tilde\bu_\epsilon, v_\epsilon$ and $h_\epsilon$ by 0 outside of $[0,T_\epsilon^\delta]$.

Since the Piola transform preserves the vanishing of the normal component of a vector field at the boundary, it can be verified that the kinematic coupling condition $\bq|_{\Gamma}\cdot \bn^{h_\epsilon} = (0,\zeta)\cdot \bn^{h_\epsilon}$ is satisfied.
Observe also, due to the fact that $v_\epsilon=(v_\epsilon)_M=0$ at $x=0,x=L$ and the boundary conditions \eqref{bc_h}, that we have, $$q_2\Big|_{\Gamma_{in/out}}=0,$$ and,
\begin{align}\label{qboundary}
    q_1\Big|_{\Gamma_{in/out}}-1=\int_{t-\kappa}^t\left((\tilde{u}_\epsilon)_1\Big|_{\Gamma_{in/out}}-1\right). 
\end{align}

Furthermore, $\nabla^h\cdot\bq=0$ and,
\begin{align}
	\|\bq\|_{L^\infty(0,T_\epsilon^\delta;\bH^{1}(\sO))} &\leq \sup_{0\le t\leq T_0}\left( \|P_t\|^2_{\bH^{1.5}(\sO)}\int_{t-\kappa}^t(\|\tilde\bu_\epsilon+\bw_\epsilon\|_{\bH^1(\sO)} +\|\bw^M_\epsilon\|_{\bH^1(\sO)} + |\mathsf{c}_\epsilon^M(s)|)ds\right) \notag\\
	&\leq \kappa^{\frac12} \|h_\epsilon\|^4_{L^\infty(0,T_0;\bH^{2}(\Gamma))} ( \|\tilde\bu_\epsilon\|_{L^2(0,T_\epsilon^\delta;\bH^1(\sO))}+\|\partial_th_\epsilon\|_{L^2(0,T_\epsilon^\delta;H^1(0,L))} )\notag\\
	&\leq C\kappa^{\frac12}.\label{boundsq}
\end{align}
Appropriate regularity of $\partial_t\bq$ and that of $\zeta$ can be verified.
Hence, $(\bq,\zeta)$ defined in \eqref{test_f}, \eqref{test_s} is a valid test function \eqref{weak_form_pen}. 

We next describe the motivation behind the choice of the test function of the form \eqref{TFform}; observe that the first term on the left hand side \eqref{weak_form_pen} containing the time derivative of $\bq$ gives us:
\begin{align*}
&\int_\kappa^{T^\delta_\epsilon}\int_{\sO}J_{h_\epsilon}\tilde\bu_\epsilon\cdot\partial_t\left( \int_{t-\kappa}^t{(\tilde\bu_\epsilon)_{\text{mod}}}-(\kappa-1)\right) \\
    &=	\int_\kappa^{T^\delta_\epsilon}\int_{\sO}J_{h_\epsilon}\tilde\bu_\epsilon\cdot\left( (\tilde\bu_\epsilon)_{\text{mod}}(t,t)-(\tilde\bu_\epsilon)_{\text{mod}}(t-\kappa,t)+\int_{t-\kappa}^t\partial_t(\tilde\bu_\epsilon)_{\text{mod}}(s,t)ds\right)\\
&=\frac12\int_\kappa^{T^\delta_\epsilon}\int_{\sO}J_{h_\epsilon}\left(|\tilde\bu_\epsilon(t)-\tilde\bu_\epsilon(t-\kappa)|^2+ |\tilde\bu_\epsilon(t)|^2 - |\tilde\bu_\epsilon(t-\kappa)|^2\right)\\
	&	+\int_\kappa^{T^\delta_\epsilon}\int_{\sO}J_{h_\epsilon}\tilde \bu_\epsilon\cdot\left( (\tilde\bu_\epsilon)_{\text{mod}}(t,t)-\tilde\bu_\epsilon(t)-(\tilde\bu_\epsilon)_{\text{mod}}(t-\kappa,t)+\tilde\bu_\epsilon(t-\kappa)+\int_{t-\kappa}^t\partial_t(\tilde\bu_\epsilon)_{\text{mod}}(s,t)ds\right).
\end{align*}
The first term on the left hand side of the equation above has the following lower bound: For some constant $C_\delta>0$ depending on $\delta$ we have
\begin{align}\label{lowerbound}
   C_\delta\int_\kappa^{T^\delta_\epsilon}\int_{\sO} |\tilde\bu_\epsilon(t)-\tilde\bu_\epsilon(t-\kappa)|^2 \leq \frac12\int_\kappa^{T^\delta_\epsilon}\int_{\sO}J_{h_\epsilon}\left( |\tilde\bu_\epsilon(t)-\tilde\bu_\epsilon(t-\kappa)|^2\right).
\end{align}
The left hand side of \eqref{lowerbound} is exactly the term (semi-norm) we must bound in terms of an appropriate power of $\kappa$ in order to prove boundedness of Nikolski norm given in \eqref{Nikolski}. Therefore, to prove \eqref{reg_temp} we will find bounds, for each term appearing in the equation above and the rest of the terms appearing in the weak form \eqref{weak_form_pen}, in terms of an appropriate power of $\kappa$.

We will only show how the penalty term appearing in the weak form \eqref{weak_form_pen} is treated. Observe, due to \eqref{qboundary}, that,
\begin{equation}
\begin{split}
   & \left| \frac1{\epsilon}\int_0^{T^\delta_\epsilon}\left(\int_{\Gamma_{out}}(\tilde u_\epsilon)_1-1\right)\left(\int_{\Gamma_{out}}q_1-1\right) \right| = \left| \frac1{\epsilon}\int_0^{T^\delta_\epsilon}\left(\int_{\Gamma_{out}}(\tilde u_\epsilon)_1-1\right)\left(\int_{t-\kappa}^t\int_{\Gamma_{out}}(\tilde u_\epsilon)_1-1\right)\right|\\
   & \qquad \qquad \qquad\leq   \frac1{\epsilon}\int_0^{T^\delta_\epsilon}\left(\int_{\Gamma_{out}}(\tilde u_\epsilon)_1-1\right)\left(\int_{0}^{T^\delta_\epsilon}\left(\int_{\Gamma_{out}}(\tilde u_\epsilon)_1-1\right)^2\right)^{\frac12}\sqrt{\kappa}\\
    &  \qquad \qquad \qquad\leq   \frac1{\epsilon}\int_{0}^{T^\delta_\epsilon}\left(\int_{\Gamma_{out}}(\tilde u_\epsilon)_1-1\right)^2\sqrt{\kappa{T^\delta_\epsilon}} \leq C\sqrt{\kappa{T^\delta_\epsilon}},
\end{split}
\end{equation}
where, thanks to the energy inequality \eqref{en_ineq_pen} the constant $C>0$ is {\it independent} of $\epsilon$.
The rest of the proof follows as in \cite{CMT26}.
\end{proof}
The purpose of proving Lemma \ref{lem:compact} is to then use the following classical compactness result:
Assume that $\mathcal{Y}_0\subset\mathcal{Y}$ are Banach spaces such that $\mathcal{Y}_0$ and $\mathcal{Y}$ are reflexive with compact embedding of $\mathcal{Y}_0$ in $\mathcal{Y}$. 
		Then for any $m>0$, the embedding	$$ 
		L^2(0,T;\mathcal{Y}_0)\cap N^{m,2}(0,T;\mathcal{Y})
		\hookrightarrow L^2(0,T;\mathcal{Y})$$	is compact.

Hence, combining Lemmas \ref{lem:compact} and the estimates \eqref{en_ineq_pen} with $\mathcal{Y}_0=\bH^1(\sO)\times H^{1}(0,L)$ and $\mathcal{Y}=\bL^2(\sO)\times L^2(0,L)$, we see that the sequence $\{(\tilde\bu_\epsilon, v_\epsilon)\}$ is relatively compact in $L^2(0,T_0;\bL^2(\sO)\times L^2(0,L))$. 
	Therefore, we obtained the following strong convergence result for fluid and structure velocities:
	\begin{proposition}\label{prop:strongbv}
		For a fixed $\delta>0$, the sequence 
		\begin{align*}
			(\tilde\bu_\epsilon, v_\epsilon) &\to (\tilde \bu, v), \quad\text{ strongly in } L^2(0,T_*^\delta;\bL^2(\sO)\times L^2(0,L)),\\
            h_{\epsilon} &\to h, \quad\text{ strongly in } L^\infty(0,T_*^\delta; C^1([0,L])),
		\end{align*}
       as $\epsilon\to 0$. Here, $T^\delta_*>0$ is the time until which the limiting structure function $h$ satisfies the minimum distance condition \eqref{hmin1}.
	\end{proposition}
    The limiting function $(\tilde\bu, h)$ is the candidate solution for the equivalent Lagrangian formulation of our problem \eqref{weak_form_gen}. Due to Proposition \ref{prop:strongbv} and the weak convergence results we obtain owing to the energy bounds \eqref{en_ineq_pen}, we can pass $\epsilon\to0$ in \eqref{weak_form_pen}. The desired local-in-time weak solution stated in Theorem \ref{thm:exist} is then given by $(\tilde\bu\circ\bA^{-1}_h,h)$. We remind the reader here that this solution is constructed on $[0,T^\delta_*]$ which is the time interval until which the structure function $h$ satisfies the minimum distance condition \eqref{hmin1}.  
    
    Now, we proceed to Step (2) described at the beginning of this section. Since the existence result above holds true for any minimum distance $\delta>0$, the time of existence $T^\delta_*$ can then be extended to some time $T_0>0$ that satisfies either one of the following conditions: 
    $$T_0 =\infty,\quad\text{or}\quad\lim_{t\to T_0}h =0.$$
    In other words, our existence result is
either global in time, or, in case the walls of the tube collapse, it holds until the time of collapse. This can be proved by using an argument identical to that presented in Theorem 3 of \cite{MC13} or that on pp. 397-398 in \cite{CDEG}. This completes the proof of Theorem \ref{thm:exist}.

\section{Regularity result: Proof of Theorem \ref{thm:regularity}}\label{sec:regularity}

The aim of this section is to prove Theorem \ref{thm:regularity}.
The main idea behind the proof is to construct a test function $\overline{\eta}$ for the structure subproblem that is equivalent to $\partial_{xx}h$. The challenge, of course then lies in constructing an appropriate corresponding test function $\bq$ for the fluid equations satisfying appropriate coupling conditions at the fluid-structure interface $\Gamma_h$. Since our sole aim in this section is to obtain a higher regularity result for $h$, our construction of test functions will not involve a tangential jump at $\Gamma_h$ i.e. we will construct a test pair that satisfies the no-slip boundary condition at the fluid-structure interface.

The rest of the section is devoted to the derivation of this {\it a priori} estimate. To justify testing \eqref{weak_form_gen} with such the test-function $\overline{\eta}= \partial_{xx}h$ rigorously, we would consider a regularization of the structure equation \eqref{beam_eqn} by augmenting the left-side with the higher order term $\epsilon \partial^6_{x}h(x,t)$ for $\epsilon>0$ and by imposing 0 boundary conditions for $\partial_{xx}h$ at $x=0,L$. The estimates we find below are all independent of $\epsilon$ as they do not involve the $H^4$-norm of $h$. Hence, our desired result \eqref{reg_spatial} can then be proved by taking the limit $\epsilon\to0$ in the regularized problem using the usual techniques. 

Firstly, we recall that we discussed the existence of a weak energy solution in Theorem \ref{thm:exist}. For this solution,
as in Section \ref{sec:exist}, we let $T^\delta_*>0$ be such that 
\begin{align}\label{hmin}
  0<  \delta\leq h(x,t),\qquad \forall x\in[0,L],\,\, t\in [0,T^\delta_*].
\end{align}
We emphasize here that we do {\bf not} seek to find bounds for $h$ independently of $\delta$ and that our aim is only to prove the additional regularity of $h$ given in \eqref{reg_spatial} stating that the structure function $h$ belongs to the space $L^2_tH^3_x$ {\it until} the time of contact. More precisely, we will prove:
\begin{align}\label{h3delta}
 \|h\|^2_{L^2(0;T^\delta_*;H^3(0,L))} + \|\partial_{xx}h(T^\delta_*)\|_{L^2(0,L)}^2\leq C_\delta +\|\partial_{xx}h_0\|_{L^2(0,L)}^2,
\end{align}
for some constant $C_\delta>0$ that depends on $\delta>0$.
\subsection{Construction of test functions} \label{test2}
To find the desired bounds for the $H^3$-norm of the structural height function $h$, we will test the structure subproblem with $\overline{\eta}= \partial_{xx}h(x,t)$. In this subsection, our aim is to then construct an appropriate test function $\overline\bfphi$ for the fluid subproblem so that  $(\overline\bfphi,\overline{\eta})$ satisfies appropriate conditions. 

We let,
\begin{align*}
 \Psi(\xi)=-2\xi^3 +3\xi^2.
\end{align*}
Then, we define $(\overline{\bfphi},\overline{\eta})$ as
\begin{align}\label{testpair1} 
\overline{\eta}(x,t)=h_{xx}(x,t),\quad \quad \overline{\bfphi} = \nabla^\perp \overline{\psi},
\end{align}
where the stream function $\overline{\psi}$ describing the incompressible flow $\overline{\bfphi}$ is given by,
\begin{equation}\label{barpsi}
    {\overline{\psi}}(x,y,t)=\partial_{x}h(x,t) \Psi\left(\frac{y}{h(x,t)}\right).
\end{equation}

The motivation behind this test pair comes from the no-slip counterpart of the present work \cite{GH_noslip}. In this paper, the authors construct a test pair such that the test function for the fluid subproblem is an approximate solution of the steady Stokes equations, defined on every `frozen' instance of the fluid domain. We stress here that our aim is only to analyze the regularity of our weak solution, given by Theorem \ref{thm:exist}. Hence, we consider the solution
on $[0,T^\delta_*]$ where the time $T^\delta_*>0$ is given in \eqref{hmin}, for each $\delta>0$ and do not seek bounds uniformly in $\delta$ (see our claim \eqref{h3delta}). Hence,
our estimates are fundamentally different in nature and do not contradict the results of \cite{GH_noslip}.
See also \cite{CMT26} where a similar `hidden' regularity result is proved for a fluid-structure interaction model in 3D, by constructing the fluid test function by solving the time-dependent Stokes equation on the reference domain and using the Piola transform composed with the Arbitrary Lagrangian-Eulerian map.

 We will now proceed to find appropriate bounds for the test pair \eqref{testpair1}. For that purpose, we first find pointwise bounds for the function $\Psi$. Hereon, we will repeatedly use the energy estimate \eqref{en_ineq} and its consequence, given in \eqref{hbounded}, that gives us boundedness in time and space of the function $h$ and that of its spatial derivative $\partial_xh$.

Notice, for any $t\leq T^\delta_*$, and for every $x\in [0,L],\frac{y}{h(x,t)}\leq 1$, that 
\begin{align}\label{psibounded}
 |\Psi\left(\frac{y}{h(x,t)}\right)| \leq C,
\end{align}
for some constant $C$ depending only on $\delta,\beta_b,\beta_s$.
Similarly, we observe the same behavior for the derivatives ${\Psi_{\xi},}\Psi_{\xi\xi},\Psi_{\xi\xi\xi}$ and obtain,
\begin{equation}\label{psi2}
\begin{split}
     \left|\Psi_{\xi}\left(\frac{y}{h(x,t)}\right)\right|+     \left|\Psi_{\xi\xi}\left(\frac{y}{h(x,t)}\right)\right|
     +\left|\Psi_{\xi\xi\xi}\left(\frac{y}{h(x,t)}\right)\right|
    \leq     C.
\end{split}
\end{equation}
Another important observation is that,
$$\Psi_{\xi}\left(1\right) =  \Psi_{\xi}(0)= 0.$$

This completes the derivation of required bounds for the derivatives of $\Psi$. Using these estimates, we will establish appropriate bounds for the fluid test function $\overline{\bfphi}$ in the next section.
\subsection{Bounds for $\overline{\bfphi}$}
In this section, we will derive the necessary bounds for the test function $\overline{\bfphi}$. All the bounds in this section are given in terms of constants that {\bf depend on $\delta>0$}, where $\delta$ is defined according to the condition \eqref{hmin}.
We recall that, by definition, we have
\begin{align*}
    &\overline{\bfphi}(x,y,t) = \left[-\partial_y\left(\partial_xh(x,t)\Psi\left(\frac{y}{h(x,t)}\right)\right), \partial_x\left(\partial_xh(x,t)\Psi\left(\frac{y}{h(x,t)}\right)\right)\right]\\
    &= \Big[-\frac{\partial_xh(x,t)}{h(x,t)} \Psi_\xi\left(\frac{y}{h(x,t)}\right),\,
    \partial_{xx}h(x,t)\Psi\left(\frac{y}{h(x,t)}\right)-\partial_xh(x,t)\left( \partial_{x}h(x,t)\Psi_{\xi}\left(\frac{y}{h(x,t)}\right)\frac{y}{h(x,t)^2}
    \right)\Big] \\
    &=: [\overline{\phi}_1(x,y), \overline{\phi}_2(x,y)].
\end{align*}
Since $\partial_{x}h=\partial_{xx}h=0$ at $x=0,L$, it can be checked easily that $\overline{\bfphi}$ satisfies the no-slip boundary conditions on $\partial F_h\setminus \Gamma_h$:
$$\overline{\bfphi}(0,y,t)=0,\quad
\overline{\bfphi}(L,y,t) = 0,\quad \overline{\bfphi}(x,0,t)=0,$$
and the following kinematic coupling condition at the moving boundary $\Gamma_h$:
$$\overline{\bfphi}|_{\Gamma_h} = (0,\overline{\eta}).$$
Next, due to \eqref{psibounded}-\eqref{psi2} and \eqref{hbounded}, we observe that for any $t\leq T^\delta_*$ and $(x,y)\in \overline F_h(t)$ we have
$$|\overline{\bfphi}(x,y,t)| \leq C_\delta\left(1+|\partial_{xx}h(x,t)|\right). $$
This estimate, along with the energy inequality \eqref{en_ineq} further implies that,
\begin{align}\label{overlinephiboundary}
    \|\overline{\bfphi}\|_{L^\infty(0,T^\delta_*; \bL^2(\Gamma_h))}\leq C_\delta.
\end{align}
Moreover, due to the one-dimensional embedding $H^1(0,L)\hookrightarrow L^\infty(0,L)$ we also have,
\begin{align}\label{overlinephilinfty}
\|\overline{\bfphi}\|_{L^2(0,T^\delta_*;\bL^\infty({F_h}))} \leq \|\partial_{xx}h\|_{L^2(0,T^\delta_*;L^\infty(0,L))} \leq C_\delta \|h\|_{L^2(0,T^\delta_*;H^3(0,L))}.
\end{align}
Next, we proceed to finding bounds for the derivative $\nabla \overline{\bfphi}$. 

We compute these spatial derivatives below:
\begin{align*}
     \partial_x \overline{\phi}_1 (x,y)&=-\partial_y\overline{\phi}_2 (x,y)  = \partial_x\left[{ -}\frac{\partial_xh(x)}{h(x)}\Psi_{\xi}\left(\frac{y}{h(x)}\right) \right]\\
       & = \frac{\partial_xh(x)}{h(x)}\left( \partial_{x}h(x) \Psi_{\xi\xi}\left(\frac{y}{h(x)}\right)\frac{y}{h(x)^2}\right)
        +\left(\frac{\partial_xh(x)^2}{h^2(x)}-\frac{\partial_{xx}h(x)}{h(x)}\right)\Psi_{\xi}\left(\frac{y}{h(x)}\right), 
\end{align*}
and
$$\partial_y\overline{\phi}_1 (x,y)
        = \partial_y\left[{ -}\frac{\partial_xh(x)}{h(x)}\Psi_{\xi}\left(\frac{y}{h(x)}\right)\right] 
        = { -}\frac{\partial_xh(x)}{h(x)^2}\Psi_{\xi\xi}\left(\frac{y}{h(x)}\right).$$

Similarly, for the second component $\overline{\phi}_2$ we find, 
\begin{align*}
     &  \partial_x \overline{\phi}_2 (x,y)
        = \partial_x\left[\partial_{xx}h(x)\Psi\left(\frac{y}{h(x)}\right) 
        - (\partial_{x}h(x))^2\Psi_{\xi}\left(\frac{y}{h(x)}\right)\frac{y}{h(x)^2} \right]\\ 
        &= \partial_{xxx}h(x)\Psi\left(\frac{y}{h(x)}\right) 
      -y\left(\frac{3\partial_xh(x)\partial_{xx}h(x)}{h(x)^2}-\frac{2(\partial_xh(x))^3}{h(x)^3}\right)\Psi_{\xi}\left(\frac{y}{h(x)}\right)
        +\frac{\partial_xh(x)^3y^2}{h(x)^4}\Psi_{\xi\xi}\left(\frac{y}{h(x)}\right)
        .
\end{align*}

Using the bounds \eqref{psibounded}-\eqref{psi2} along with \eqref{hbounded} and \eqref{hmin}, we obtain for any $t\leq T^\delta_*$ and $(x,y)\in\overline{F}_h$ that
$$|\nabla\overline{\bfphi}(x,y,t)| \leq C_\delta(1+|\partial_{xx}h(x)|+|\partial_{xxx}h(x)|),$$
which tells us that
\begin{align}\label{overlinephiH1}
    \|\nabla\overline{\bfphi}\|_{L^2(0,T^\delta_*;\bL^2(F_h))} \leq C_\delta\|h\|_{L^2(0,T^\delta_*;H^3(0,L))}.
\end{align}

Next, we will apply the test pair $(\overline{\bfphi},\overline{\eta})$, defined in \eqref{testpair1}, in the weak form \eqref{weak_form_gen}. This yields,
\begin{equation} \label{weak_form_exist}
\begin{split}
& \int\limits_{F_{h_0}} \bu_0\cdot\overline{\bfphi}(0)=\int\limits_{F_h(T^\delta_*)} \bu(T^\delta_*)\cdot\overline{\bfphi}(T^\delta_*)
    -\int\limits_0^{T^\delta_*}\int\limits_{{F_h(t)}}\bu\cdot\partial_t \overline{\bfphi}   + \mu\int\limits_0^{T^\delta_*}\int\limits_{{F_h(t)}}\nabla\bu:\nabla\overline{\bfphi}\\
  & \int\limits_0^{T^\delta_*}\int\limits_{{F_h(t)}}  (\bu\cdot\nabla)\bu\cdot\overline{\bfphi}-\int\limits_0^{T^\delta_*}\int\limits_{\Gamma_h(t)}(\bu\cdot\overline{\bfphi})(\bu\cdot \bn^{h})+\int\limits_0^{T^\delta_*}\int\limits_{0}^{L}\left(\alpha \partial_{xx}h\partial_{xx}\overline{\eta} + \gamma\partial_{xt}h\partial_{x}\overline{\eta}-\partial_th\partial_{t}\overline{\eta} \right)\\
&:= J_1+...+J_6.
    \end{split}
\end{equation}
We will analyze each term $J_i; i=1,..,6$ starting with $J_2$ which is one of the most critical terms:
$$   \int\limits_0^{T^\delta_*}  \int\limits_{{F_h(t)}} \partial_t\overline{\bfphi} \cdot \bu 
=  \int\limits_0^{T^\delta_*}  \int\limits_{{F_h(t)}} \partial_t\overline{\phi}_1 u_1 
+  \int\limits_0^{T^\delta_*}  \int\limits_{{F_h(t)}} \partial_t\overline{\phi}_2 u_2 .$$
Observe that, due to the fact that $\|\partial_th\|_{L^2(0,T;H^1_0(0,L))}\leq C$, the 1D embedding $H^1(0,L)\hookrightarrow L^\infty(0,L)$ and that we have the lower bounds \eqref{hmin} for $h$, we have
$$\|\overline{\phi}_1\|_{L^2(0,T^\delta_*;L^2(F_h))} \leq C_\delta\|\partial_th\|_{L^2(0,T^\delta_*;H^1(0,L))}\leq C_\delta. $$
The same is not true of $\partial_t\overline{\phi}_2$.
Hence, to handle the second term we define the following anti-derivative,
\begin{equation*}
    P_x \partial_t\overline{\phi}_2(x,y,t) := \int\limits_0^x \partial_t\overline{\phi}_2(s,y,t)ds =\int\limits_0^x \partial_x\partial_t\overline{\psi}(s,y,t)ds = \partial_t\overline{\psi}(x,y,t)-\partial_t\overline{\psi}(0,y,t)=\partial_t\overline{\psi}(x,y,t),
\end{equation*}
where we used the fact that $\partial_xh(0,t)=\partial_th(0,t)=0$.
By definition we thus have,
\begin{equation*}
    \partial_x(P_x \partial_t\overline{\phi}_2) = \partial_t\overline{\phi}_2,
\end{equation*}
and so we substitute this into the integral and integrate by parts. Since $P_x\partial_t\bar\phi_2=0$ at the inlet boundary $x=0$ and the outlet boundary $x=L$, we obtain
\begin{align*}
  \int\limits_0^{T^\delta_*}   \int\limits_{{F_h(t)}} \partial_t\overline{\phi}_2 u_2=   \int\limits_0^{T^\delta_*}  \int\limits_{{F_h(t)}}\partial_x( P_x \partial_t\overline{\phi}_2)u_2
   = -  \int\limits_0^{T^\delta_*}  \int\limits_{{F_h(t)}} P_x\partial_t\overline{\phi}_2 \partial_x u_2
    +   \int\limits_0^{T^\delta_*}  \int\limits_{\Gamma_h(t)} \left(P_x \partial_t\overline{\phi}_2\right) u_2 n^x ,
\end{align*}
where $n^x=\frac{-\partial_xh}{\sqrt{1+\partial_xh^2}}$ is the horizontal component of the normal $\bn^h$ at the moving boundary $\Gamma_h$.

To treat the first term on the right-hand side we first observe, for $\overline{\psi}$ defined in \eqref{barpsi}, that,
$$|\partial_t\overline{\psi}| \leq C_\delta (|\partial_{xt}h|+|\partial_th|),$$
where $C_\delta>0$ depends on $\delta$ defined by the condition \eqref{hmin}. Hence, the first term on the right-hand side
is bounded, due to \eqref{en_ineq}, as follows,
\begin{align*}
   \int\limits_0^{T^\delta_*}  \int\limits_{{F_h(t)}} P_x \partial_t\overline{\phi}_2 \partial_x u_2 &\leq \|P_x \partial_t\overline{\phi}_2\|_{L^2(0,T^\delta_*;L^2(F_h))}\| \partial_x u _2\|_{L^2(0,T^\delta_*;L^2(F_h))}\\
   &=
   \|\partial_t\overline{\psi}\|_{L^2(0,T^\delta_*;L^2(F_h))}\| \partial_x u _2\|_{L^2(0,T^\delta_*;L^2(F_h))}
   \leq C.
\end{align*}
Now for the second term, by using the fact that $\Psi_{\xi}(1)=0$, we obtain that
\begin{align*}
    \partial_t\overline{\psi} (x,h(x,t))= \partial_{xt}h(x,t).
\end{align*}
Thus, since $n^x=\frac{-\partial_xh}{\sqrt{1+\partial_xh^2}}$ is bounded \eqref{normtangbound}, we have, thanks to \eqref{en_ineq}, that
\begin{align*}
   \int\limits_0^{T^\delta_*}  \int\limits_{\Gamma_{h}} \left(P_x \partial_t\overline{\phi}_2\right) u_2 n^x &\leq \int\limits_0^{T^\delta_*}\|u_2\|_{L^2(\Gamma_h)}\|\partial_t\overline{\psi}\|_{L^2(\Gamma_h)}\\
   &\leq C\|\bu\|_{L^2(0,T^\delta_*;\bL^2(\Gamma_h))}\|\partial_{tx}h\|_{L^2(0,T^\delta_*;L^2(0,L))} \leq C.
\end{align*}
Now we will continue with the rest of the fluid terms $J_3$ and $J_4$. For the dissipation term, using \eqref{overlinephiH1}, we obtain
\begin{align*}
&|J_3|=\left|\int_0^{T^\delta_*}\int_{{F_h(t)}}\nabla\bu:\nabla\overline{\bfphi}\right|    \leq C\|\nabla\bu\|_{L^2(0,T^\delta_*;\bL^2({F_h}))}\|\nabla\overline{\bfphi}\|_{L^2(0,T^\delta_*;\bL^2({F_h}))}\leq C+\frac{\alpha}{4}\|h\|^2_{L^2(0,T^\delta_*;H^3(0,L))}. 
\end{align*}
For the advection term $J_4$ we use \eqref{overlinephilinfty} to obtain,
\begin{align*}
\left| \int\limits_0^{T^\delta_*}\int\limits_{{F_h(t)}}  (\bu\cdot\nabla)\bu\cdot\overline{\bfphi}\right|
   & \leq C\|\bu\|_{L^\infty(0,T^\delta_*;\bL^2({F_h}))}\|\bu\|_{L^2(0,T^\delta_*;\bH^1({F_h}))}\|\overline{\bfphi}\|_{L^2(0,T^\delta_*;\bL^\infty({F_h}))} \\
   &\leq C+\frac{\alpha}{4}\|h\|^2_{L^2(0,T^\delta_*;H^3(0,L))}.
\end{align*}

Next, we will consider the boundary integral $J_5$.
 Note that the impermeability condition states that $\displaystyle \bu|_{\Gamma_h}\cdot \bn^h=(0,\partial_th)\cdot \bn^h$. Hence, by applying \eqref{overlinephiboundary}, we obtain
\begin{align*}
    |J_5| \leq \|\bu\|_{L^2(0,T^\delta_*;\bL^2(\Gamma_h))}\|\overline{\bfphi}\|_{L^\infty(0,T^\delta_*;\bL^2(\Gamma_h))}\|\partial_th\|_{L^2(0,T^\delta_*;H^1(0,L))} \leq C.
\end{align*}
Hence, after collecting all the bounds for the terms $J_i;i=1,...,5$ and applying integrating by parts to the structure terms $J_6$, we deduce, for some constant $C_\delta>0$ depending on $\delta>0$, that
\begin{align}\label{h3inequality}
 \|h\|^2_{L^2(0;T^\delta_*;H^3(0,L))} + \|\partial_{xx}h(T^\delta_*)\|_{L^2(0,L)}^2\leq C_\delta +\|\partial_{xx}h_0\|_{L^2(0,L)}^2.
\end{align}
Since the bound in \eqref{h3inequality}, holds for every $\delta>0$,
 we obtain the desired result \eqref{reg_spatial} until the time of existence $T_0>0$ given by Theorem \ref{thm:exist}. In other words, the result \eqref{reg_spatial} holds until the time of contact $T_0>0$ of the compliant top boundary with the bottom boundary of the tube. This completes the proof of Theorem \ref{thm:regularity}.

\section{Finite-time contact: Proof of Theorem \ref{mainresult}}\label{sec:contact}
In this section we will prove our main result stated in Theorem \ref{mainresult}. To do so, we will assume, for a {\bf contradiction} that the structure satisfies no-collapse condition. That is, for any $T>0$, 
\begin{align}\label{hypothesis}
\text{assume that $h(x,t) >0$ for all $t\in [0,T]$ and $x\in [0,L]$.}
\end{align}

To find an appropriate contradiction, our proof relies on the construction of a test function pair $(\bfphi,\eta)$ for the weak formulation \eqref{weak_form_gen}. The proof is divided into several steps: \\
(1) We first construct a test function pair $(\bfphi,\eta)\in\cal{T}(0,T)$, such that $\bfphi$ additionally satisfies the outflow condition \eqref{phioutlet} and approximates the solution to the steady Stokes equation,\\
(2) Then we find appropriate bounds for this test function and show that it is a valid test function for the weak form \eqref{weak_form_gen},\\
(3) Then, we will estimate each term appearing in \eqref{weak_form_gen}, corresponding to the constructed test pair, in terms of $T>0$.

We begin with Step (1) in the following subsection.

\subsection{Construction of test functions with slip}\label{test1} 
The aim of this subsection is to construct an appropriate pair of test functions $(\bfphi,\eta)$ for the fluid-structure system that will give us the desired contradiction. 
We will list here some important characteristics of the construction. Our main idea is based on the works of Gerard-Varet and Hillairet and Wang \cite{GH_slip}: that is, we aim to use the solution to the time-independent Stokes equation {with slip boundary conditions} solved on the fluid domain frozen at any time $t\geq 0$ as our fluid test function $\bfphi$. This naturally leads us to seek the minimizer of the associated energy functional, %
\begin{equation}
    \cal{E}_h(\bfphi, \bfeta) :=\int\limits_{{F_h}}|\nabla \bfphi|^2 
    +\int\limits_{\Gamma_h} \left[  \left(\frac{1}{\beta_s} + \kappa(h)\right)|(\bfphi-\bfeta)\cdot \bftau^h|^2 \right] + \frac{1}{\beta_b}\int\limits_{\Gamma_b}|\bfphi \cdot \bftau^b|^2.
    \label{en_func}
\end{equation}
As is also mentioned in \cite{GH_slip}, it is a non-trivial task to find the exact solution to the Stokes equations in terms of the structure displacement $h$. Hence, we must find a suitable approximation for the solution to the Stokes equation and thus to that of the energy functional \eqref{en_func} which is explicitly solvable in terms of $h$.
To find such an energy functional and its corresponding minimizer, we first breakdown the behavior of the desired solution.

Since, we are operating in two spatial dimensions, it is convenient to write the desired solenoidal test function $\bfphi$ in terms of a stream function $\psi$,
\begin{equation}\label{streamfunction}
    \bfphi=\nabla^{\perp}\psi=(-\partial_y \psi, \partial_x\psi).
\end{equation}
We will now take a closer look at the desired behavior of $(\bfphi,\eta)$ at the moving interface $\Gamma_h$. 
Recall that a valid fluid test function $\bfphi$ is required to satisfy impermeability condition \eqref{bottom_bc} on the bottom boundary $\Gamma_b$. This condition translates into $\partial_x \psi(x,0) =0$. Hence, we set, 
\begin{equation}
    \psi(x,0) = 0.
    \label{stream_bot}
\end{equation}
Similarly, the impermeability condition on the top boundary $\Gamma_h$, which gives us the continuity of normal velocities at the interface $\Gamma_h$, reads $\displaystyle\bfphi(x,h(x,t),t)\cdot \bn^h(x,t) = \eta(x,t)\cdot \bn^h(x,t)$. Hence, for $\bfeta=(0,\eta)$ we have,
$$\nabla^{\perp}\psi { (x,h(x,t),t)}\cdot \bn^h = \nabla\psi{ (x,h(x,t),t)}\cdot\bftau^h   = \bfeta(x,t)\cdot \bn^h=\frac{\eta}{\sqrt{1+(\partial_x h(x,t))^2}},$$
where $\bn^h$ and $\bftau^h$ are defined in \eqref{normaltangent}.

For simplicity, we denote  $\xi(x)=(x,h(x,t),t)$, and using this notation observe that
\begin{align}\label{dpsi}
    \frac{d}{dx}(\psi(\xi(x)) = \nabla\psi(\xi(x))\cdot\xi'(x) = \nabla\psi(\xi(x)) \cdot (1, \partial_xh(x)) = \sqrt{1+(\partial_x h(x))^2}(\nabla\psi(\xi(x))\cdot\bftau^h).
\end{align}
Hence, by combining the two identities above we obtain
$$\frac{d\psi(\xi(x))}{dx}=\eta(x).$$
We impose the corner condition $\psi(0,H,t)=H$ and integrate the equation above in $x$. This tells us that the stream function given in \eqref{streamfunction} must satisfy, 
\begin{equation}
    \psi(x,h(x,t),t)=\int\limits_0^x \eta(\zeta,t) d\zeta + H.
    \label{stream_top}
\end{equation}

Hence, we conclude that for the fluid-structure test functions pair $(\bfphi,\eta)$ to satisfy the impermeability conditions on the top and the bottom boundaries,
the stream function given in \eqref{streamfunction} must satisfy \eqref{stream_bot} and \eqref{stream_top} on the bottom and the top boundaries, respectively.

Next, using these two equations, we characterize the behavior of the tangential jumps, appearing in the energy functional \eqref{en_func}, at top and the bottom boundaries. 

We begin with the time-dependent top boundary $\Gamma_h$. Observe that,
\begin{equation*}
    \Big(\bfphi{(x,h(x,t),t)}-\bfeta(x,t)\Big)\cdot\bftau^h = \frac{-\partial_y\psi(x,h(x,t),t)+\partial_x h(x,t)(\partial_x\psi(x,h(x,t),t)-\eta(x,t))}{\sqrt{1+(\partial_x h(x,t))^2}},
\end{equation*}
where $\bn^h$ and $\bftau^h$ are defined in \eqref{normaltangent}.

 Now we recall \eqref{dpsi} and restate it below,
\begin{equation*}
    \frac{d\psi{(x, h(x,t))}}{dx}= \partial_x\psi(x, h(x,t))+\partial_xh(x,t)\partial_y\psi(x, h(x,t)) = \eta(x,t).
\end{equation*}
Notice that, this gives us on $\Gamma_h$, that
\begin{equation}\label{jumptop}
 (\bfphi(x, h(x,t))-\bfeta(x,t))\cdot\bftau^h 
 = -\partial_y\psi (x, h(x,t))\sqrt{(1+(\partial_x h(x,t))^2)}.
\end{equation}
Similarly, on the bottom boundary $\Gamma_b$ we obtain, 
\begin{equation}\label{jumpbottom}
    \phi_2(x,0,t) = -\partial_y\psi(x,0,t).
\end{equation}
We then finally substitute \eqref{jumptop}-\eqref{jumpbottom} in the energy functional \eqref{en_func} for the fluid-structure test pair $(\bfphi,\eta)$ in terms of the stream function $\psi$ defined in \eqref{streamfunction} that satisfies \eqref{stream_bot}-\eqref{stream_top}, which gives us that
\begin{equation}
    {\cal{E}}_h(\bfphi,\eta)
    :=\int\limits_{{F_h(t)}}|\nabla \bfphi|^2 
    + \int\limits_{\Gamma_h}\left[ \left(\frac{1}{\beta_s} + \kappa(h)\right) \left| \partial_y\psi\right|^2 \left|(1+(\partial_x h)^2)^{3/2}\right|\right]
    + \frac{1}{\beta_b}\int\limits_{\Gamma_b}\left|\partial_y\psi\right|^2,
    \label{en_func_inter}
\end{equation}
where $\kappa(h)$ is the curvature of $\Gamma_h$.

For any fixed $\eta$, to find a suitable approximation of the energy functional \eqref{en_func_inter}, we exploit ideas that are generally applied in thin-gap geometries. In the narrow region between the boundaries, we suppose that the vertical scale is much smaller than the horizontal scale. Consequently, variations of the flow in the horizontal direction are expected to be much weaker than the variations in the vertical direction. This motivates a `slow-$x$ approximation', in which the derivatives in the $x$-direction are neglected to leading order. Under this approximation, at each fixed horizontal position $x$, the flow profiled is approximated by a function of the vertical coordinate, $y$, only. We denote this approximation for the stream function by $\psi^x = \psi^x(y)$ and obtain the following reduced energy functional 
\begin{equation}
    \widetilde{\cal{E}}_h(\psi^x)
    = \int\limits_0^{h(x)}|(\psi^x)''(y)|^2dy 
    + \alpha_s(h(x))|(\psi^x)'(h(x))|^2
    + \alpha_b |(\psi^x)'(0)|^2,
\label{approx_en_func}
\end{equation}
where 
\begin{equation}    \label{alpha}
    \alpha_s(h(x))=\left(\frac{1}{\beta_s}+\kappa(h(x))\right)(1+\partial_xh^2(x))^{3/2}
\quad\text{ and }\quad
    \alpha_b = \frac{1}{\beta_b}.
\end{equation}
This reduced functional \eqref{approx_en_func} can be interpreted as the vertical dissipation energy associated with the approximate Stokes profile at a fixed horizontal location $x$. A similar functional is obtained in \cite{GH_slip}. The reason we use it is that it allows the minimization problem to become explicitly solvable, allowing us to construct a convenient approximate solution to the Stokes equations. We minimize \eqref{approx_en_func} with the Dirichlet boundary conditions \eqref{stream_bot} and \eqref{stream_top}. This leads us to the following form: 
\begin{equation}\label{def:streamfunction}
    {\psi}(x,y,t)=\left(\int\limits_0^x\eta(\zeta,t) d\zeta+H\right) \Phi\left(x,\frac{y}{h(x,t)},t\right),
\end{equation}
where $\Phi(x,\xi,t)$, satisfying the associated Euler-Lagrange equations,  is a cubic polynomial in $\xi$ given by: 
\begin{equation}\label{def:Phi}
        \Phi(x,\xi,t) = a_h(x,t)\xi^3 + b_h(x,t)\xi^2 + c( x,t)\xi .
\end{equation}
This equation is supplemented with the following boundary conditions at the top and the bottom boundaries of the reference domain,
\begin{align*}
      \Phi(x,0,t) =0,\quad
        \Phi_{\xi\xi}(x,0,t)-\lambda_b\Phi_{\xi}(x,0,t)=0 ,\\
          \Phi(x,1,t) =1, \quad
        \Phi_{\xi\xi}(x,1,t)+\lambda_s\Phi_\xi(x,1,t)=0. 
\end{align*}
Solving the system \eqref{def:Phi} for the boundary conditions above, we obtain the coefficients as follows:
\begin{equation}\label{def:abc}
    \begin{cases}
    a_h(x,t)=-\frac{2(\lambda_s + \lambda_b + \lambda_s\lambda_b)}{\lambda_s\lambda_b + 4(\lambda_s + \lambda_b) +12}(x,t), \\
    b_h(x,t)=\frac{3\lambda_b(\lambda_s+2)}{\lambda_s\lambda_b + 4(\lambda_s + \lambda_b) +12}(x,t),\\
    c_h(x,t)= \frac{6(\lambda_s+2)}{\lambda_s\lambda_b + 4(\lambda_s + \lambda_b) +12}(x,t),
\end{cases}
\end{equation}
where the functions $\lambda_s,\lambda_b$ are chosen to be,
\begin{equation}\label{def:lambdas}
        \lambda_s (x,t) = \frac{h(x,t)}{\beta_s}, \quad\text{ and}\quad 
        \lambda_b (x,t)=\frac{h(x,t)}{\beta_b}.
\end{equation}
We note here that while the minimization problem leads to a different set of values for $\lambda_s,\lambda_b$, some simplifications lead to \eqref{def:lambdas}, which are sufficient for our purposes.

We {\bf summarize} that, by construction, the fluid test function $\bfphi = \nabla^\perp \psi$, is given, for any $t\geq$, $(x,y)\in\overline{F_h(t)}$, by 
\begin{equation}\label{def:phi}
    \bfphi(x,y,t) = \left[-\partial_y\left(h(x,t)\Phi\left(x,\frac{y}{h(x,t)},t\right)\right), \partial_x\left(h(x,t)\Phi\left(x,\frac{y}{h(x,t)},t\right)\right)\right],
\end{equation}
where $\Phi$ is defined in \eqref{def:Phi}. We recall that the stream-function $\psi$, defined in \eqref{def:streamfunction} satisfies the boundary conditions \eqref{stream_bot} and \eqref{stream_top}.

Now, we strategically choose the structure test function $\eta$ as,
\begin{align}\label{def:eta}
    \eta (x,t) = \partial_x h(x,t).
\end{align}
Note, thanks to Theorem \ref{thm:regularity}, that this is valid choice of test function for the structure subproblem.
The main motivation behind this choice is that, after spatial integration (see \eqref{def:streamfunction}) the function $\eta$ yields $h$ which is used to counter/cancel the factors of $h$ that appear in the denominator of the test function $\bfphi$ and its derivatives.

Recall that, by construction, we know that the impermeability conditions 
$\bfphi|_{\Gamma_h}\cdot \bn^h=(0,\eta)\cdot \bn^h$ and $\phi_2|_{\Gamma_b}=0$ are satisfied.
We will next evaluate the behavior of this test function $\bfphi$ at $\Gamma_{in/out}$, the inlet and the outlet parts  of the fixed boundary.

For that purpose, we introduce the following notation:
\begin{equation*}
    \Phi_i = \partial_{v_i}\Phi(v_1,v_2), \quad i=1,2,
\end{equation*}
where for $i=1$ we take the derivative with respect to the first variable ($v_1$), and likewise when $i=2$ we take derivative with respect to the second variable ($v_2$). 

For simplicity, we temporarily suppress the notation showing the dependence of $\bfphi$ on $t$ in several places and using the above notation we write, 
\begin{align*}
    \bfphi(x,y) 
    &= \Big[-h(x) \Phi_2\left(x,\frac{y}{h(x)}\right)\frac{1}{h(x)},\\
    &\hspace{1in}
    \partial_{x}h(x)\Phi\left(x,\frac{y}{h(x)}\right)+h(x)\left( \Phi_1\left(x,\frac{y}{h(x)}\right)  + \partial_{x}h(x)\Phi_2\left(x,\frac{y}{h(x)}\right)\frac{-y}{h(x)^2}
    \right)\Big] \\
    &=\left[ -\Phi_2\left(x,\frac{y}{h(x)}\right),
    \partial_{x}h(x)\Phi\left(x,\frac{y}{h(x)}\right) + h(x)\Phi_1\left(x,\frac{y}{h(x)}\right) - \partial_{x}h(x)\Phi_2\left(x,\frac{y}{h(x)}\right)\frac{y}{h(x)}
    \right] \\
    &= [\phi_1(x,y), \phi_2(x,y)].
\end{align*}
Note that at the inlet and the outlet boundaries, due to \eqref{bc_h}, we have
$$\bfphi(0,y,t)= \left(-\Phi_2\left(0,\frac{y}H\right),H\Phi_1\left(0,\frac{y}H\right) \right),\quad\text{ and }\quad \bfphi(L,y,t)= \left(-\Phi_2\left(L,\frac{y}H\right),H\Phi_1\left(L,\frac{y}H\right) \right).$$
We observe that the functions $a,b,c$ are periodic in $x$: 
\begin{equation}\label{abc_periodic}
   \begin{split}
        a(0,t)=a(L,t)&  = -2\frac{\beta_sH+\beta_bH+H^2}{12\beta_s\beta_b+4(\beta_s+\beta_b)H+H^2},\\
    b(0,t)=b(L,t)& = 3\frac{2\beta_sH+H^2}{12\beta_s\beta_b+4(\beta_s+\beta_b)H+H^2}\\ 
    c(0,t) = c(L,t)&= 6\frac{\beta_bH+ 2\beta_s\beta_b}{12\beta_s\beta_b+4(\beta_s+\beta_b)H+H^2}.\\
   \end{split}
\end{equation}
Using $\phi_1(0,y,t)=\phi_1(L,y,t)=-\Phi_2(0,y/H,t)$,
we also evaluate 
$$-\int_0^H\Phi_2(0,y/H,t)dy= -H\int_0^1\Phi_2(0,y,t)dy= -H(\Phi(0,1,t)-\Phi(0,0,t))=-H.$$
Moreover, we have that $\displaystyle\Phi_1\left(0,\frac{y}H\right) =\Phi_1\left(L,\frac{y}H\right)=0 $ because
\begin{align}\label{boundaryabc_derivative}
     \partial_x a(0,t)=\partial_x a(L,t)=0,\quad     
     \partial_x b(0,t)=\partial_x b(L,t)=0,\quad 
     \partial_x c(0,t)=\partial_x c(L,t)=0.
\end{align}
Hence, by construction we have
$$\int_0^H\phi_1(0,y,t)dy=\int_0^H\phi_1(L,y,t)dy=-H,\quad\text{ and } \quad\phi_2(0,y,t)=\phi_2(L,y,t)=0.$$

{Hence, $\bfphi$ as constructed satisfies the required boundary conditions at $\Gamma_{in/out}$ and $\Gamma_b$.}

Next, we will establish the regularity of the pair $(\bfphi,\eta)$ thus proving its validity as a test function pair for the weak formulation \eqref{weak_form_gen}. For that purpose, in the next section we will first obtain appropriate bounds for the function $\Phi$ defined in \eqref{def:Phi}. 

\subsection{Bounds for $\Phi$}\label{sec:Phibounds}
In this subsection we list the behavior of the derivatives of $\Phi$, which depend on the derivatives of $a_h(x,t), b_h(x,t), c_h(x,t)$. These terms are needed to obtain bounds in particular for the fluid test function $\bfphi$ constructed in the previous section in \eqref{def:phi}. 

We emphasize that, throughout the remainder of the paper, unless otherwise specified, $C>0$ will denote a generic constant depending only on the prescribed data $\beta_s, \beta_b, \mathbf{u}_0, h_0, v_0, P_{in}, P_{out}$, and independent of both the functions under consideration and the time horizon $T$.

Recall that we use the following notation
\begin{equation*}
    \Phi_i = \partial_{v_i}\Phi(v_1,v_2), \quad i=1,2.
\end{equation*}

First we summarize our findings 
that gives detailed calculations for the derivatives of the functions $a_h(x,t),b_h(x,t),c_h(x,t)$ defined in \eqref{def:abc}.
Thanks to \eqref{hbounded} we find, for some constant $C>0$ depending only on the slip-lengths $\beta_s,\beta_b$, that for any $x\in[0,L], t\in[0,T]$, we have
\begin{equation}\label{abc_behavior}
    \begin{split}
 &   |a_h(x,t)| \leq Ch(x,t), \quad  |b(x,t)| \leq Ch(x,t), \quad |c(x,t)| \leq C, \\
  & |\partial_x a_h(x,t)|+ |\partial_x b_h(x,t)| +| \partial_x c_h(x,t)| \leq  C|\partial_x h(x,t)|,\\
      & |\partial_{xx} a_h(x,t)|+ |\partial_{xx} b_h(x,t)| +| \partial_{xx} c_h(x,t)| \leq C(1+ |\partial_{xx} h(x,t)|),\\
    & |\partial_{xt} a_h(x,t)|+ |\partial_{xt} b_h(x,t)| +| \partial_{xt} c_h(x,t)| \leq C(|\partial_th(x,t)|+| \partial_{xt} h(x,t)|).
 \end{split}
\end{equation}
   
Using these computations and the energy estimate \eqref{en_ineq}, we will now proceed to find bounds for the derivatives of the function $\Phi$.

Notice that for any $t\geq 0$, we have for every $x\in [0,L],\frac{y}{h(x,t)}\leq 1$, that
\begin{align*}
 |\Phi\left(x,\frac{y}{h(x,t)},t\right)| \leq |a_h(x,t)|+|b_h(x,t)|+|c_h(x,t)| \leq C(1+h(x,t)),
\end{align*}
for some constant $C$ depending only on $\beta_b,\beta_s$.
Thanks to \eqref{hbounded} we can then conclude, for some constant $C>0$ depending only on the given pressure data, initial data and the slip coefficients $\beta_s,\beta_b$, that
\begin{align}\label{phi_bound}
     \sup_{(x,y)\in \overline{F_h(t)},\, t\geq 0}  |\Phi\left(x,\frac{y}{h(x,t)},t\right)| \leq C .
\end{align}
Similarly, we observe the same behavior for the derivatives $\Phi_2,\Phi_{22},\Phi_{222}$ and obtain,
\begin{equation}
\begin{split}\label{phi_bound2h}
    &\left|\Phi_2\left(x,\frac{y}{h(x,t)},t\right)\right|
     \lesssim       |a_h(x,t)|+|b_h(x,t)|+|c_h(x,t)|\leq C      \left[ 1+h(x,t)\right] \leq C, \\
 &  \left[ 
     \left|\Phi_{22}\left(x,\frac{y}{h(x,t)},t\right)\right|
     +\left|\Phi_{222}\left(x,\frac{y}{h(x,t)},t\right)\right|
     \right]   \lesssim       |a_h(x,t)|+|b_h(x,t)|\leq     C     h(x,t).
\end{split}
\end{equation}

Next, due to \eqref{abc_behavior} and the bounds \eqref{hbounded} we find the following bounds for the derivatives $\Phi$,
\begin{equation}\label{phi_bound1}
    \begin{split}
           \sup_{(x,y)\in \overline{F_h(t)},\, t\geq 0} & \left[ \left|\Phi_1\left(x,\frac{y}{h(x,t)},t\right)\right|
     +\left|\Phi_{12}\left(x,\frac{y}{h(x,t)},t\right)\right|
     +\left|\Phi_{122}\left(x,\frac{y}{h(x,t)},t\right)\right|\right]\\
     & \leq      \sup_{(x,y)\in \overline{F_h(t)},\, t\geq 0}  |\partial_{x}a_h(x,t)|+|\partial_{x}b_h(x,t)|+|\partial_{x}c_h(x,t)|\\
     &\leq      \sup_{(x,y)\in \overline{F_h(t)},\, t\geq 0}  C      |\partial_xh(x,t)| \leq C.
    \end{split}
\end{equation}
Next, using \eqref{abc_behavior} we find, for any $(x,y)\in {F_h(t)}, t\geq 0$, that
\begin{equation}\label{phi_bound11}
    \begin{split}
        \left[ \left|\Phi_{11}\left(x,\frac{y}{h(x,t)},t\right)\right|
     +\left|\Phi_{112}\left(x,\frac{y}{h(x,t)},t\right)\right|\right]&
     \leq     |\partial_{xx}a_h(x,t)|+|\partial_{xx}b_h(x,t)|+|\partial_{xx}c_h(x,t)|\\
     &\leq     C   \left( 1 + |\partial_{xx} h(x,t)|\right).
    \end{split}
\end{equation}
 
This completes the derivation of bounds for the derivatives of $\Phi$. Using these estimates, we will establish appropriate bounds for the fluid test function $\bfphi$ in the next section.
\subsection{Bounds for the fluid test function $\bfphi$}\label{sec:phibounds}
The main aim of this section is to find appropriate bounds for $\bfphi$ and prove the validity of $\bfphi$ as a fluid test function. These bounds will be employed in the following subsection, Section \ref{sec:conclusion6}, where we will apply the test function and estimate each term appearing in the weak form.

We recall that our definition of $\bfphi$, by construction, is,
\begin{align*}
    \bfphi(x,y,t) 
     &=\left[ -\Phi_2\left(x,\frac{y}{h(x)}\right),
    \partial_{x}h(x)\Phi\left(x,\frac{y}{h(x)}\right) + h(x)\Phi_1\left(x,\frac{y}{h(x)}\right) - \partial_{x}h(x)\Phi_2\left(x,\frac{y}{h(x)}\right)\frac{y}{h(x)}
    \right] \\
    &= [\phi_1(x,y), \phi_2(x,y)],
\end{align*}
where we recall that $\Phi$ is defined in \eqref{def:Phi} in terms of the functions $a,b,c$ given in \eqref{def:abc}.

We first note that, due to \eqref{phi_bound}, \eqref{phi_bound2h}, and \eqref{phi_bound1}, we have
\begin{align}\label{phi_boundC}
      \sup_{(x,y)\in \overline{F_h(t)},\, t\geq 0} |\bfphi(x,y,t)|\leq C.
\end{align}

To find bounds for $\|\nabla \bfphi\|_{L^2({F_h(t)})}$ we first compute the spatial derivatives of the first component $\phi_1$:
\begin{equation}
\begin{split}\label{phi1_derivative}
     \partial_x \phi_1 (x,y)=- \partial_y\phi_2 (x,y)=
        &= \partial_x\left[{ -}\Phi_2\left(x,\frac{y}{h(x)}\right) \right]\\
       & = { -}\Phi_{21}\left(x,\frac{y}{h(x)}\right) + \partial_{x}h(x) \Phi_{22}\left(x,\frac{y}{h(x)}\right)\frac{y}{h(x)^2},\\
\partial_y\phi_1 (x,y)
        &= \partial_y\left[{ -}\Phi_2\left(x,\frac{y}{h(x)}\right)\right] 
        = { -}\Phi_{22}\left(x,\frac{y}{h(x)}\right)\frac{1}{h(x)}.
\end{split}
\end{equation}
Next, we compute the spatial derivatives of the second component $\phi_2$ and obtain,
\begin{align*}
       \partial_x \phi_2 (x,y)
        &= \partial_x\left[\partial_{x}h(x)\Phi\left(x,\frac{y}{h(x)}\right) 
        + h(x)\Phi_1\left(x,\frac{y}{h(x)}\right) 
        - \partial_{x}h(x)\Phi_2\left(x,\frac{y}{h(x)}\right)\frac{y}{h(x)} \right]\\
        &= \partial_{xx}h(x)\Phi\left(x,\frac{y}{h(x)}\right) 
        + 2\partial_{x}h(x)\left( \Phi_1\left(x,\frac{y}{h(x)}\right) \right)\\
        &\quad 
        + h(x) \Phi_{11}\left(x,\frac{y}{h(x)}\right) 
        - \partial_{x}h(x)\Phi_{12}\left(x,\frac{y}{h(x)}\right)\frac{y}{h(x)}-\partial_{xx}h(x)\Phi_2\left(x,\frac{y}{h(x)}\right) \frac{y}{h(x)} \\
        &\quad -
        \partial_{x}h(x)\left(\Phi_{21}\left(x,\frac{y}{h(x)}\right) 
        - \partial_{x}h(x)\Phi_{22}\left(x,\frac{y}{h(x)}\right)\frac{y}{h(x)^2}\right)\frac{y}{h(x)}.
\end{align*}

Now we will apply the bounds \eqref{phi_bound}-\eqref{phi_bound11} for the derivatives of $\Phi$ found in the previous section, and the Lipschitz bounds \eqref{hbounded} for $h$ to find the necessary pointwise bounds for the derivatives of $\bfphi$.  

Observe that due to \eqref{phi_bound2h}$_2$, \eqref{phi_bound1} and \eqref{hbounded} we obtain, for some constant $C>0$, that
\begin{align}\label{phi1derivatives}
    &|\partial_x\phi_1(x,y)| \leq C, 
    \quad
    |\partial_y\phi_1(x,y)| \leq  C,\qquad\text{ for any }(x,y)\in \overline{F_h(t)}, t\geq 0.
\end{align}
Furthermore, using \eqref{phi_bound2h}$_2$, \eqref{phi_bound1}, and \eqref{phi_bound11}  we obtain
\begin{align}\label{phi2derivatives}
       |\partial_x\phi_2(x,y)| \leq C(1+|\partial_{xx}h(x)|) 
       ,\qquad
    |\partial_y\phi_2(x,y)| \leq C,
    \quad\text{ for any }(x,y)\in \overline{F_h(t)}, t\geq 0.
\end{align}
We collect all the bounds above and apply the energy inequality \eqref{en_ineq} to obtain the desired estimate:
\begin{equation}\label{phi_supH1}
    \sup_{t\in [0,T]} \|\nabla \bfphi\|^2_{\bL^2({F_h(t)})} \leq C\sup_{t\in [0,T]} \int_0^L(1+|\partial_{xx}h(x,t)|^2) dx \leq C,
\end{equation}
where the constant $C>0$, independent of $T$, depends only on given data and the Lipschitz constant of $h$ as obtained in \eqref{hbounded}.

Next, we will prove that
\begin{align}\label{phit}
     \|\partial_t \bfphi\|_{L^1(0,T;\bL^2(F_h))} &\leq C ,
\end{align}
where the constant $C>0$ is independent of $T$.

For that purpose, we will find $\partial_t\bfphi$, which leads to the following straightforward yet cumbersome computations.
For the first component we obtain
\begin{align*}
    \partial_t \phi_1(x,y,t) &= 
    -\frac{3 y^2 \partial_{t} h(x,t) \partial_x a_h(x,t)}{h(x,t)^2}
    +\frac{6 y^2 \partial_{t} h(x,t) a_h(x,t)}{h(x,t)^3}
    -\frac{2 y
   \partial_{t} h(x,t) \partial_x b_h(x,t)}{h(x,t)}\\&
   +\frac{2 y \partial_{t} h(x,t) b_h(x,t)}{h(x,t)^2}
   -\partial_{t} h(x,t) \partial_x c_h(x,t).
\end{align*}
Now using \eqref{abc_behavior} and the $C^1$ bounds for $h$ in \eqref{hbounded}, we obtain 
\begin{align*}
    |\partial_t \phi_1| \leq C 
    |1+\partial_t h(x,t)|.
\end{align*}

For the second component we obtain
\begin{align*}
    \partial_t \phi_2 (x,y,t)&= \frac{y^3
   \partial_{t} h(x,t) \partial_{x} h(x,t) \partial_{xt} a_h(x,t)}{h(x,t)^2}
   -\frac{4 y^3 \partial_{t} h(x,t) \partial_{x} h(x,t)
   \partial_x a_h(x,t)}{h(x,t)^3}
   +\frac{y^3 \partial_{xt} h(x,t) \partial_x a_h(x,t)}{h(x,t)^2}\\&
   +\frac{6 y^3 \partial_{t} h(x,t) \partial_{x} h(x,t)
   a_h(x,t)}{h(x,t)^4}
   -\frac{2 y^3 \partial_{xt} h(x,t) a_h(x,t)}{h(x,t)^3}
   +\frac{y^2 \partial_{t} h(x,t) \partial_{x} h(x,t)
   \partial_{xt} b_h(x,t)}{h(x,t)}\\&
   -\frac{2 y^2 \partial_{t} h(x,t) \partial_{x} h(x,t) \partial_x b_h(x,t)}{h(x,t)^2}
   +\frac{y^2 \partial_{xt} h(x,t)
   \partial_x b_h(x,t)}{h(x,t)}
   +\frac{2 y^2 \partial_{t} h(x,t) \partial_{x} h(x,t) b_h(x,t)}{h(x,t)^3}\\&
   -\frac{y^2 \partial_{xt} h(x,t)
   b_h(x,t)}{h(x,t)^2}+y \partial_{t} h(x,t) \partial_{x} h(x,t) \partial_{xt} c_h(x,t)+y \partial_{xt} h(x,t) \partial_x c_h(x,t).
\end{align*}
We then use \eqref{abc_behavior} and \eqref{hbounded} to obtain 
\begin{align*}
    |\partial_t \phi_2| &\leq C 
    \left[ 
        |\partial_t h(x,t) ( \partial_t h(x,t) + \partial_{xt} h(x,t))| 
        +|\partial_t h(x,t) | 
        + | \partial_{xt} h(x,t)| 
    \right]
\end{align*}
Hence, using the embedding $H^1(0,L)\hookrightarrow L^4(0,L)$ in one-dimension and applying the energy inequality \eqref{en_ineq} we conclude that
\begin{align}\label{phi_t}
     \|\partial_t \bfphi\|_{L^1(0,T;\bL^2(F_h))} &\leq C \left(
     1
     +\|\partial_th\|^2_{L^2(0,T;H^1_0(0,L))}\right) \leq C,
\end{align}
where the constant $C>0$ does not depend on $T$.

\subsection{Conclusion of the proof}\label{sec:conclusion6}
 Recall that, due to Theorem \ref{thm:regularity}, we know that $\eta=\partial_xh\in L^2(0,T;H^2_0(0,L))\cap H^1(0,T;L^2(0,L))$ and that $\bfphi$ satisfies the bounds \eqref{phi_boundC}-\eqref{phit} along with appropriate boundary conditions. Furthermore, $\bfphi|_{\Gamma_h}\cdot\bn^h = (0,\eta)\cdot \bn^h$. Hence, we can use 
$(\bfphi,\eta)$,  defined in \eqref{def:phi}-\eqref{def:eta} as a test pair for the weak formulation \eqref{weak_form_gen} (technically, we must multiply these functions with the constant $\frac{-1}H$).
This yields,
\begin{multline} \label{weak_form}
\left[  \int\limits_{F_h(t)} \bu(t)\cdot\bfphi(t) +\int_0^L\partial_th(t)\zeta(t) \right]_{t=0}^{t=T}
    -\int\limits_0^T\int\limits_{{F_h(t)}}\bu\cdot\partial_t \bfphi   + 2\int\limits_0^T\int\limits_{{F_h(t)}}\nabla\bu:\nabla\bfphi
    \\+\frac12\left[ \int\limits_0^T\int\limits_{{F_h(t)}} \left[ (\bu\cdot\nabla)\bu\cdot\bfphi- (\bu\cdot\nabla)\bfphi\cdot\bu\right] -\int\limits_0^T\int\limits_{\Gamma_h(t)}(\bu\cdot\bfphi)(\bu\cdot \bn^h)\right] \\
   +\frac{1}{\beta_b}\int\limits_0^T\int\limits_{\Gamma_b}(\bu\cdot\bftau^b)\cdot (\bfphi\cdot\bftau^b) 
        + \frac{1}{\beta_s}\int\limits_0^T\int\limits_{\Gamma_h(t)}((\bu-(0,\partial_{t}h))\cdot\bftau^h)\cdot ((\bfphi-(0,\partial_xh))\cdot\bftau^h)\\
       +\int\limits_0^T\int\limits_{0}^{L}\Big(
     \alpha \partial_{xx}h\partial_{xxx}h + \gamma\partial_{xt}h\partial_{xx}h-\partial_th\partial_{tx}h \Big)=\int\limits_0^T \int\limits_{\Gamma_{in}} P_{in}\phi_1 
	-\int\limits_0^T \int\limits_{\Gamma_{out}}  P_{out}\phi_1.
\end{multline}
We write \eqref{weak_form} as
\begin{align}\label{weakformshort}
    \int_0^T P_{in}(t)\left(\int_{\Gamma_{in}} \phi_1(t) \right) dt
	-\int_0^T P_{out}(t)\left(\int_{\Gamma_{out}} \phi_1(t) \right) dt:= \sum_{i=1}^7 I_{i}(T).
\end{align}

Before continuing, we make the following observations for the left-hand side pressure terms.
As observed earlier, by construction, we have $\phi_1(0,y,t)=\phi_1(L,y,t)=-\Phi_2(0,y/H,t)$.
We evaluate 
$$-\int_0^H\Phi_2(0,y/H,t)dy= -H\int_0^1\Phi_2(0,y,t)dy= -H(\Phi(0,1,t)-\Phi(0,0,t))=-H.$$
We apply the assumption \eqref{pressure_difference} and find the following lower bound for the left hand side term of \eqref{weakformshort}:
\begin{align}\label{pressure_lowerbound}
    p_0TH \leq \int\limits_0^T \left(\left(P_{in}(t)-P_{out}(t)\right)\int\limits_{\Gamma_{in/out}}   \phi_1(t)\right)dt.
\end{align}
To come to the desired contradiction, we will now prove, for some constant $C>0$ independent of $T$, that
\begin{align}\label{claim_I}
    |\sum_{i=1}^7 I_{i}(T)| \leq C\sqrt{T}.
\end{align}

In what follows we will find bounds of the form \eqref{claim_I} for each term $I_i; i=1,...,7$ appearing \eqref{weak_form}. 

\noindent{\bf Time derivative term $I_2$:}\\
We first consider the term $I_2$ in \eqref{weak_form}. Using the energy estimate \eqref{en_ineq} and the bounds on the time derivative of $\bfphi$ given in \eqref{phi_t}, we immediately obtain
\begin{align*}
    |I_2(T)|=|\int\limits_0^T\int\limits_{{F_h(t)}} \bu \cdot \partial_t \bfphi  |   \leq C
    \|\bu\|_{L^\infty(0,T;\bL^2({F_h(t)}))} \|\partial_t \bfphi\|_{L^1(0,T;\bL^2({F_h(t)}))}\leq C.
\end{align*} 
Similarly, the first term $I_1$ is treated similarly using the energy bounds and \eqref{phi_boundC}.

\noindent{\bf Fluid dissipation term $I_3$:}\\
Next, we consider the term $I_3$ in \eqref{weak_form}.
\begin{align*}
&\left|\int_0^T\int_{{F_h(t)}}\nabla\bu:\nabla\bfphi\right|    \leq C\|\nabla\bu\|_{L^2(0,T;\bL^2({F_h}))}\|\nabla\bfphi\|_{L^2(0,T;\bL^2({F_h}))}. 
\end{align*}
Hence, thanks to \eqref{en_ineq} and \eqref{phi_supH1} we obtain for some constant $C>0$ independent of $T$ that,
\begin{align}
    \label{term_dissipation}
       | I_3(T)| \leq C\sqrt{T}.
\end{align}

\medskip

\noindent{\bf Advection term $I_4$:}\\
Next we consider the term $I_4$. An application \eqref{en_ineq}, and \eqref{phi_boundC}, yields
\begin{align*}
   | I_{4,1}(T)|= \left| \int\limits_0^T\int\limits_{{F_h(t)}}  (\bu\cdot\nabla)\bu\cdot\bfphi
    \right|
    \leq C\|\bu\|_{L^\infty(0,T;\bL^2({F_h}))}
    \|\bu\|_{L^2(0,T;\bH^1({F_h}))}
    \|\bfphi\|_{L^\infty(0,T;\bL^\infty({F_h}))} \leq C\sqrt{T}.
\end{align*}
For some universal $C$ independent of time.
Hence, thanks to \eqref{en_ineq} and \eqref{phi_supH1} we obtain for some constant $C>0$ independent of $T$ that,
\begin{align*}
       | I_{4,1}(T)| \leq C\sqrt{T}.
\end{align*}
Next, we will consider the boundary term: $\displaystyle I_{4,2}:=\int\limits_0^T\int\limits_{\Gamma_h(t)}(\bu\cdot\bfphi)(\bu\cdot \bn^h)$.
 Note that the impermeability condition states that $\displaystyle \bu|_{\Gamma_h}\cdot \bn^h=(0,\partial_th)\cdot \bn^h$. Hence, by applying the inequalities \eqref{en_ineq}, \eqref{phi_boundC}, and   \eqref{uboundary}, we obtain
\begin{align*}
    |I_{4,2}(T)| \leq \|\bu\|_{L^2(0,T;\bL^2(\Gamma_h))}\|\bfphi\|_{L^\infty(0,T;\bL^2(\Gamma_h))}\|\partial_th\|_{L^2(0,T;H^1(0,L))} \leq C.
\end{align*}
The remaining term is treated similarly using the following trace inequality at the inlet and outlet boundaries:
For any $(x,y)\in \sO:= (0,L)\times(0,1)$:
$$\tilde\bu(x,y) = h(x)^2\bu(x,yh(x)).$$
For this definition we have, for $i=1,2$, that,
\begin{align*}
&\partial_x\tilde{u}_i(x,y)=2h(x)\partial_xh(x)u_i(x,yh(x))+h(x)^2(\partial_xu_i(x,yh(x))+\partial_yu_i (x,yh(x))\partial_xh(x)y),\\ &\partial_y\tilde{u}_i(x,y)=h(x)^3\partial_yu_i(x,yh(x)) + 2h(x)\partial_xh(x)u_i(x,yh(x)).
\end{align*}
Hence, for this definition, using \eqref{hbounded} we have for any $(x,y)\in \sO$, that
\begin{align*}
    |\nabla\tilde\bu(x,y)| \leq \sup_{[0,L]}\Big(|\partial_xh|+|h|\Big)h(x)|\nabla\bu(x,yh(x))|\leq C h(x)|\nabla\bu(x,yh(x))|.
\end{align*}
Hence, for any $1\leq p<\infty$, using \eqref{hbounded} again, we obtain, 
\begin{align*}
    \|\nabla\tilde\bu\|^p_{\bL^p(\sO)}& \leq \int_0^L\int_0^1h(x)|\nabla\bu(x,yh(x))|^pdxdy= \int_0^L\int_0^{h(x)}|h(x)|^{p-1}|\nabla\bu(x,y)|^pdxdy\\
    &\leq C_{p,\sO}\|\nabla\bu\|_{\bL^p(F_h)}^p.
\end{align*}
Now, for any $1\leq p<\infty$, we apply the standard trace inequality on $\sO=(0,L)\times (0,1)$ and, thus obtain a constant $C_{\sO,p}>0$, independent of $h$ and $t$, such that
\begin{equation}\label{trace0}
    \begin{split}
H^{2-\frac1p} \|\bu(0,\cdot)\|_{L^p(0,H)}=\|\tilde\bu(0,\cdot)\|_{L^p(0,1)} &\leq C_{\sO,p}\left(\|\tilde\bu\|_{L^p(\sO)}+ \|\nabla\tilde\bu\|_{L^p(\sO)}\right)
 \\& \leq C_{\sO,p} \|\bu\|_{{\bf W}^{1,p}(F_h)}.
    \end{split}
\end{equation}
Hence, we apply \eqref{trace0} with $p=2$, the energy bounds \eqref{en_ineq}, \eqref{hbounded} and the bounds for $\bfphi$ given in \eqref{phi_boundC} as follows,
\begin{align*}
    \int\limits_0^T\int\limits_{{F_h(t)}} (\bu\cdot\nabla)\bfphi\cdot\bu &=
    \Big|\int\limits_0^T\int\limits_{{F_h(t)}} -(\bu\cdot\nabla)\bu\cdot\bfphi 
    +\int\limits_0^T\int\limits_{\Gamma_h(t)}(\bu\cdot\bq)(\bu\cdot \bn^{h}) -\int\limits_0^T\int_{\Gamma_{in}}|u_1|^2q_1+\int\limits_0^T\int_{\Gamma_{out}}|u_1|^2q_1\Big|\\
    & \leq |I_{4,1}(T)| + |I_{4,2}(T)| + \|\bu\|^{2}_{L^2(0,T;\bH^1(F_h))}\|\bfphi\|_{L^\infty([0,T]\times\Gamma_{in/out})}\\
    &\leq C\sqrt{T}.
\end{align*}
Hence, 
\begin{align}    \label{term_advection}
       | I_{4}(T)| \leq C\sqrt{T}.
\end{align}

\noindent{\bf Boundary terms with jump $I_5,I_6$}:\\
Next we consider the boundary integrals $I_5,I_6$ in \eqref{weak_form}:
$$ I_5+I_6=\frac{1}{\beta_b}\int\limits_0^T\int\limits_{\Gamma_b}(\bu\cdot\bftau^b)\cdot (\bfphi\cdot\bftau^b) 
        + \frac{1}{\beta_s}\int\limits_0^T\int\limits_{\Gamma_h(t)}((\bu-(0,\partial_{t}h))\cdot\bftau^h)\cdot ((\bfphi-(0,\partial_xh))\cdot\bftau^h).$$
We will first consider the term $I_6$. Recall that from \eqref{jumptop} we have
\begin{equation*}
 (\bfphi(x, h(x,t))-\bfeta(x,t))\cdot\bftau^h 
 = \phi_1(x, h(x,t))\sqrt{(1+(\partial_x h(x,t))^2)}.
\end{equation*}
Hence, thanks to \eqref{en_ineq},\eqref{hbounded}, and \eqref{phi_boundC}, we obtain
\begin{align*}
   &\left| \int\limits_0^T\int\limits_{\Gamma_h(t)}((\bu-(0,\partial_{t}h))\cdot\bftau^h)\cdot ((\bfphi-(0,\partial_xh))\cdot\bftau^h)\right|\\
  &\leq C\sqrt{T} \|(\bu-(0,\partial_th))\cdot\bftau^h\|_{L^2(0,T;L^2(\Gamma_h))}\|\phi_1\|_{L^\infty(0,T;L^2(\Gamma_h))} \leq C\sqrt{T}.
\end{align*}
The term $I_6$ is treated identically.

\noindent{\bf Structure dissipation terms $I_7$:}\\
To bound the final term $I_7$ in the weak form \eqref{weak_form}, we apply the energy estimate \eqref{en_ineq} and the additional structure regularity result \eqref{reg_spatial} as follows, 
\begin{align*}
&\int\limits_0^T\int\limits_0^L\partial_{xx}h(x,t)\partial_{xxx}h(x,t) = \frac12 \int\limits_0^T |\partial_{xx}h(L)|^2-|\partial_{xx}h(0)|^2=0,\\
&|\int\limits_0^T\int\limits_0^L\partial_th(x,t)\partial_{tx}h(x,t)|
   \leq C\|\partial_t h\|_{L^2(0,T;H^1(0,L))}^2 \leq C,\\
 &  \int\limits_0^T\int\limits_0^L \partial_{xt}h(x,t) \partial_{xx}h(x,t)  
 \leq C\|\partial_t h\|_{L^2(0,T;H^1(0,L))}\|h\|_{L^2(0,T;H^2(0,L))} \leq C\sqrt{T}.
\end{align*}

We have thus proved the desired bound \eqref{claim_I}. 

Combining \eqref{claim_I} with \eqref{pressure_lowerbound}, we obtain, 
\begin{align}
      p_0TH \leq \int\limits_0^T \left(\left(P_{in}(t)-P_{out}(t)\right)\int\limits_{\Gamma_{in/out}}   \phi_1(t)\right)dt \leq    |\sum_{i=1}^9 I_{i}(T)| \leq C\sqrt{T},
   \end{align}
which leads us to the desired contradiction of \eqref{hypothesis} for large $T$.
This completes the proof of Theorem \ref{mainresult}.

\section*{Acknowledgment}
The authors would like to acknowledge the support by the Applied Mathematics Department at the University of Washington provided for Nash Ward, and the support by the National Science Foundation grant DMS-2553666 (transferred from DMS-2407197) awarded to Krutika Tawri.

\bibliographystyle{plain}
\bibliography{Nash}

@article {CDEG,
    AUTHOR = {Chambolle, A. and Desjardins, B. and Esteban, M.
              J. and Grandmont, C.},
     TITLE = {Existence of weak solutions for the unsteady interaction of a
              viscous fluid with an elastic plate},
   JOURNAL = {J. Math. Fluid Mech.},
  FJOURNAL = {Journal of Mathematical Fluid Mechanics},
    VOLUME = {7},
      YEAR = {2005},
    NUMBER = {3},
     PAGES = {368--404},
      ISSN = {1422-6928},
   MRCLASS = {35Q35 (35D05 35Q30 74F10 76D03)},
  MRNUMBER = {2166981},
MRREVIEWER = {Changxing Miao},
}

@article {KTZ10,
    AUTHOR = {Kukavica, I. and Tuffaha, A. and Ziane, M.},
     TITLE = {Strong solutions for a fluid structure interaction system},
   JOURNAL = {Adv. Differential Equations},
  FJOURNAL = {Advances in Differential Equations},
    VOLUME = {15},
      YEAR = {2010},
    NUMBER = {3-4},
     PAGES = {231--254},
      ISSN = {1079-9389},
   MRCLASS = {35Q35 (74F10 76D05)},
  MRNUMBER = {2588449},
MRREVIEWER = {Michael\ J.\ Carley},
}

@incollection {F03,
    AUTHOR = {Feireisl, E.},
     TITLE = {On the motion of rigid bodies in a viscous incompressible
              fluid},
      NOTE = {Dedicated to Philippe B\'enilan},
   JOURNAL = {J. Evol. Equ.},
  FJOURNAL = {Journal of Evolution Equations},
    VOLUME = {3},
      YEAR = {2003},
    NUMBER = {3},
     PAGES = {419--441},
      ISSN = {1424-3199,1424-3202},
   MRCLASS = {35Q30 (76D03 76D05)},
  MRNUMBER = {2019028},
MRREVIEWER = {Luigi\ Carlo\ Berselli},
}

@article {SST02,
    AUTHOR = {San Mart\'in, J. A. and Starovoitov, V. and
              Tucsnak, M.},
     TITLE = {Global weak solutions for the two-dimensional motion of
              several rigid bodies in an incompressible viscous fluid},
   JOURNAL = {Arch. Ration. Mech. Anal.},
  FJOURNAL = {Archive for Rational Mechanics and Analysis},
    VOLUME = {161},
      YEAR = {2002},
    NUMBER = {2},
     PAGES = {113--147},
      ISSN = {0003-9527,1432-0673},
   MRCLASS = {35Q35 (35D05 35Q30 35Q72 76D03)},
  MRNUMBER = {1870954},
MRREVIEWER = {Denis\ Serre},
       DOI = {10.1007/s002050100172},
       URL = {https://doi.org/10.1007/s002050100172},
}

@article {CS10,
    AUTHOR = {Cheng, C. H. A. and Shkoller, S.},
     TITLE = {The interaction of the 3{D} {N}avier-{S}tokes equations with a
              moving nonlinear {K}oiter elastic shell},
   JOURNAL = {SIAM J. Math. Anal.},
  FJOURNAL = {SIAM Journal on Mathematical Analysis},
    VOLUME = {42},
      YEAR = {2010},
    NUMBER = {3},
     PAGES = {1094--1155},
      ISSN = {0036-1410,1095-7154},
   MRCLASS = {74F10 (35Q30 35Q74 74H20 74H25 74K25 76D05)},
  MRNUMBER = {2644917},
MRREVIEWER = {Natalia\ B.\ Chinchaladze},
    
}

@article {CS06,
    AUTHOR = {Coutand, D. and Shkoller, S.},
     TITLE = {The interaction between quasilinear elastodynamics and the
              {N}avier-{S}tokes equations},
   JOURNAL = {Arch. Ration. Mech. Anal.},
  FJOURNAL = {Archive for Rational Mechanics and Analysis},
    VOLUME = {179},
      YEAR = {2006},
    NUMBER = {3},
     PAGES = {303--352},
      ISSN = {0003-9527,1432-0673},
   MRCLASS = {74F10 (35Q30 35Q72 74B20 74H20 74H25 76D05)},
  MRNUMBER = {2208319},
MRREVIEWER = {Beno\^{i}t\ P.\ Desjardins},
}

@article {LR14,
    AUTHOR = {Lengeler, D. and Ru\v{z}i\v{c}ka, M.},
     TITLE = {Weak solutions for an incompressible {N}ewtonian fluid
              interacting with a {K}oiter type shell},
   JOURNAL = {Arch. Ration. Mech. Anal.},
  FJOURNAL = {Archive for Rational Mechanics and Analysis},
    VOLUME = {211},
      YEAR = {2014},
    NUMBER = {1},
     PAGES = {205--255},
      ISSN = {0003-9527,1432-0673},
   MRCLASS = {35Q35 (35D30 35Q74 74F10 74K25 76D05)},
  MRNUMBER = {3147436},
MRREVIEWER = {Zhaoyin\ Xiang},
}

@article {MC16,
    AUTHOR = {Muha, B. and {\v C}ani\'c, S.},
     TITLE = {Existence of a weak solution to a fluid-elastic structure
              interaction problem with the {N}avier slip boundary condition},
   JOURNAL = {J. Differential Equations},
  FJOURNAL = {Journal of Differential Equations},
    VOLUME = {260},
      YEAR = {2016},
    NUMBER = {12},
     PAGES = {8550--8589},
}

@article{MC19,
    AUTHOR = {Muha, B. and \v{C}ani\'{c}, S.},
     TITLE = {A generalization of the {A}ubin-{L}ions-{S}imon compactness
              lemma for problems on moving domains},
   JOURNAL = {J. Differential Equations},
  FJOURNAL = {Journal of Differential Equations},
    VOLUME = {266},
      YEAR = {2019},
    NUMBER = {12},
     PAGES = {8370--8418},
      ISSN = {0022-0396},
   MRCLASS = {35Q30 (35B45 35K90 35R37)},
  MRNUMBER = {3944259},
MRREVIEWER = {Michael J. Carley},
}

@misc{BKLM25,
      title={A no-contact result for a plate-fluid interaction system in dimension three}, 
      author={M. Bukal and I. Kukavica and L. Li and B. Muha},
      year={2025},
      note={arXiv:2510.09992},
      archivePrefix={},
      primaryClass={math.AP},
     
}

@misc{MR26,
      title={Existence of weak solutions for incompressible fluid-Koiter shell interactions with Navier slip boundary condition}, 
      author={C. Mindril\v{a} and A. Roy},
      year={2026},
      note={arXiv:2602.20016},
      archivePrefix={},
      primaryClass={math.AP},
}

@misc{CKS26,
      title={Fluid-Structure interactions with Navier- and full-slip boundary conditions}, 
      author={A. Cesik and M. Kampschulte and S. Schwarzacher},
      year={2026},
      note={arXiv:2603.12030},
      archivePrefix={},
      primaryClass={math.AP},
     
}

@article {T24,
    AUTHOR = {Tawri, K.},
     TITLE = {A stochastic fluid-structure interaction problem with the
              {N}avier-slip boundary condition},
   JOURNAL = {SIAM J. Math. Anal.},
  FJOURNAL = {SIAM Journal on Mathematical Analysis},
    VOLUME = {56},
      YEAR = {2024},
    NUMBER = {6},
     PAGES = {7508--7544},
      ISSN = {0036-1410,1095-7154},
   MRCLASS = {60H15 (35A01 35D30 35Q30 35R60 74F10)},
  MRNUMBER = {4823178},
MRREVIEWER = {Martin\ Ondrej\'at},
}

@article {MC13,
    AUTHOR = {Muha, B. and  \v{C}ani\'{c}, S.},
     TITLE = {Existence of a weak solution to a nonlinear fluid-structure
              interaction problem modeling the flow of an incompressible,
              viscous fluid in a cylinder with deformable walls},
   JOURNAL = {Arch. Ration. Mech. Anal.},
  FJOURNAL = {Archive for Rational Mechanics and Analysis},
    VOLUME = {207},
      YEAR = {2013},
    NUMBER = {3},
     PAGES = {919--968},
      ISSN = {0003-9527},
   MRCLASS = {35Q35 (35D30 76D05 76Z05 92C35)},
  MRNUMBER = {3017292},
MRREVIEWER = {Nader Masmoudi},
}

@incollection {NP10,
    AUTHOR = {Neustupa, J. and Penel, P},
     TITLE = {A weak solvability of the {N}avier-{S}tokes equation with
              {N}avier's boundary condition around a ball striking the wall},
 BOOKTITLE = {Advances in mathematical fluid mechanics},
     PAGES = {385--407},
 PUBLISHER = {Springer, Berlin},
      YEAR = {2010},
      ISBN = {978-3-642-04067-2},
   MRCLASS = {35Q30 (35D30 76D03 76D05)},
  MRNUMBER = {2665044},
MRREVIEWER = {Isabelle\ Gruais},
}

@article {CMT26,
  author = {\v{C}ani\'{c}, S. and Muha, B. and Tawri, K.},
  title = {Existence and Regularity Results for a Nonlinear Fluid-structure Interaction Problem with Three-dimensional Structural Displacement},
  journal = {to appear in {SIAM J. Math. Anal.}},
  note= {preprint at arXiv:2409.06939},
  year = {2026},
  }

@book {CKMT25,
    AUTHOR = {\v{C}ani\'{c}, S. and Kuan, J. and Muha, B. and
              Tawri, K.},
     TITLE = {Deterministic and {S}tochastic {F}luid-{S}tructure
              {I}nteraction},
    SERIES = {Advances in Mathematical Fluid Mechanics},
 PUBLISHER = {Birkh\"auser/Springer, Cham},
      YEAR = {2025},
     PAGES = {xvii+613},
      ISBN = {978-3-032-00897-8; 978-3-032-00898-5},
   MRCLASS = {74-01 (35 74B05 74F10 74Hxx 74S60 76Dxx)},
  MRNUMBER = {5033898},
}

@article {GH12,
    AUTHOR = {G\'erard-Varet, D. and Hillairet, M.},
     TITLE = {Computation of the drag force on a sphere close to a wall: the
              roughness issue},
   JOURNAL = {ESAIM Math. Model. Numer. Anal.},
  FJOURNAL = {ESAIM. Mathematical Modelling and Numerical Analysis},
    VOLUME = {46},
      YEAR = {2012},
    NUMBER = {5},
     PAGES = {1201--1224},
      ISSN = {2822-7840,2804-7214},
   MRCLASS = {76D07 (35Q35 74F10)},
  MRNUMBER = {2916378},
MRREVIEWER = {Mark\ Thompson},
       DOI = {10.1051/m2an/2012001},
       URL = {https://doi.org/10.1051/m2an/2012001},
}

@article {GH_slip,
    AUTHOR = {G\'erard-Varet, David and Hillairet, Matthieu and Wang, Chao},
     TITLE = {The influence of boundary conditions on the contact problem in
              a 3{D} {N}avier-{S}tokes flow},
   JOURNAL = {J. Math. Pures Appl. (9)},
  FJOURNAL = {Journal de Math\'ematiques Pures et Appliqu\'ees. Neuvi\`eme
              S\'erie},
    VOLUME = {103},
      YEAR = {2015},
    NUMBER = {1},
     PAGES = {1--38},
      ISSN = {0021-7824,1776-3371},
   MRCLASS = {35Q35 (35B44 74F10 76D03 76D05)},
  MRNUMBER = {3281946},
MRREVIEWER = {Byungsoo\ Moon},
}

@article{GH_noslip,
    title={Existence of global strong solutions to a beam–fluid interaction system}, 
    volume={220}, 
    DOI={10.1007/s00205-015-0954-y}, 
    number={3}, 
     JOURNAL = {Arch. Ration. Mech. Anal.},
  FJOURNAL = {Archive for Rational Mechanics and Analysis},
    author={Grandmont, C. and Hillairet, M.}, 
    year={2016}, 
    pages={1283-1333}}

@article {GH14,
    AUTHOR = {G\'erard-Varet, D. and Hillairet, M.},
     TITLE = {Existence of weak solutions up to collision for viscous
              fluid-solid systems with slip},
   JOURNAL = {Comm. Pure Appl. Math.},
  FJOURNAL = {Communications on Pure and Applied Mathematics},
    VOLUME = {67},
      YEAR = {2014},
    NUMBER = {12},
     PAGES = {2022--2075},
      ISSN = {0010-3640,1097-0312},
   MRCLASS = {35Q30 (35D30 74F10 76D05)},
  MRNUMBER = {3272367},
MRREVIEWER = {Emil\ Wiedemann},
}

@incollection {S04,
    AUTHOR = {Starovoitov, V.},
     TITLE = {Behavior of a rigid body in an incompressible viscous fluid
              near a boundary},
 BOOKTITLE = {Free boundary problems ({T}rento, 2002)},
    SERIES = {Internat. Ser. Numer. Math.},
    VOLUME = {147},
     PAGES = {313--327},
 PUBLISHER = {Birkh\"auser, Basel},
      YEAR = {2004},
      ISBN = {3-7643-2193-8},
   MRCLASS = {76D99 (35Q35)},
  MRNUMBER = {2044583},
MRREVIEWER = {Beno\^it\ P.\ Desjardins},
}

@article {GSST22,
    AUTHOR = {Gravina, G. and Schwarzacher, S. and Sou\v{c}ek, O. and Tůma, K.},
     TITLE = {Contactless rebound of elastic bodies in a viscous
              incompressible fluid},
   JOURNAL = {J. Fluid Mech.},
  FJOURNAL = {Journal of Fluid Mechanics},
    VOLUME = {942},
      YEAR = {2022},
     PAGES = {Paper No. A34, 46},
      ISSN = {0022-1120,1469-7645},
   MRCLASS = {76D08 (76T20)},
  MRNUMBER = {4427109},
}

@article {BFFG22,
    AUTHOR = {Burman, E. and Fern\'andez, M and Frei, S. and
              Gerosa, F.},
     TITLE = {A mechanically consistent model for fluid-structure
              interactions with contact including seepage},
   JOURNAL = {Comput. Methods Appl. Mech. Engrg.},
  FJOURNAL = {Computer Methods in Applied Mechanics and Engineering},
    VOLUME = {392},
      YEAR = {2022},
     PAGES = {Paper No. 114637, 28},
      ISSN = {0045-7825,1879-2138},
   MRCLASS = {74F10 (74M15 76S05)},
  MRNUMBER = {4383066},
MRREVIEWER = {Maia\ M.\ Svanadze},
}

@article {CGH21,
    AUTHOR = {Casanova, J. and Grandmont, C. and
              Hillairet, M.},
     TITLE = {On an existence theory for a fluid-beam problem encompassing
              possible contacts},
   JOURNAL = {J. \'Ec. polytech. Math.},
  FJOURNAL = {Journal de l'\'Ecole polytechnique. Math\'ematiques},
    VOLUME = {8},
      YEAR = {2021},
     PAGES = {933--971},
      ISSN = {2429-7100,2270-518X},
   MRCLASS = {35Q35 (74F10 74K10 76D03 76D05)},
  MRNUMBER = {4237023},
}

@book {Galdi,
    AUTHOR = {Galdi, G. P.},
     TITLE = {An introduction to the mathematical theory of the
              {N}avier-{S}tokes equations},
    SERIES = {Springer Monographs in Mathematics},
   EDITION = {Second},
      NOTE = {Steady-state problems},
 PUBLISHER = {Springer, New York},
      YEAR = {2011},
     PAGES = {xiv+1018},
      ISBN = {978-0-387-09619-3},
   MRCLASS = {35Q30 (35-02 76D03 76D05 76D07)},
  MRNUMBER = {2808162},
}

@article {G08,
    AUTHOR = {Grandmont, C.},
     TITLE = {Existence of weak solutions for the unsteady interaction of a
              viscous fluid with an elastic plate},
   JOURNAL = {SIAM J. Math. Anal.},
  FJOURNAL = {SIAM Journal on Mathematical Analysis},
    VOLUME = {40},
      YEAR = {2008},
    NUMBER = {2},
     PAGES = {716--737},
      ISSN = {0036-1410,1095-7154},
   MRCLASS = {35Q35 (35D05 74F10 76D03)},
  MRNUMBER = {2438783},
MRREVIEWER = {Dmitry\ A.\ Vorotnikov},
}

@article {S23,
    AUTHOR = {Sperone, G.},
     TITLE = {Homogenization of the steady-state {N}avier-{S}tokes equations
              with prescribed flux rate or pressure drop in a perforated
              pipe},
   JOURNAL = {J. Differential Equations},
  FJOURNAL = {Journal of Differential Equations},
    VOLUME = {375},
      YEAR = {2023},
     PAGES = {1--29},
      ISSN = {0022-0396,1090-2732},
   MRCLASS = {76M50 (35B27 35M12 35Q31 76D05)},
  MRNUMBER = {4642964},
MRREVIEWER = {Florian\ Oschmann},
}

@article {HRT96,
    AUTHOR = {Heywood, J. G. and Rannacher, R. and Turek, S.},
     TITLE = {Artificial boundaries and flux and pressure conditions for the
              incompressible {N}avier-{S}tokes equations},
   JOURNAL = {Internat. J. Numer. Methods Fluids},
  FJOURNAL = {International Journal for Numerical Methods in Fluids},
    VOLUME = {22},
      YEAR = {1996},
    NUMBER = {5},
     PAGES = {325--352},
      ISSN = {0271-2091,1097-0363},
   MRCLASS = {76M10 (76D05 76M25)},
  MRNUMBER = {1380844},
MRREVIEWER = {Vitoriano\ Ruas},
}

\end{document}